\documentclass{amsart}%
\usepackage{amssymb}
\usepackage{amsfonts}
\usepackage{amsmath}
\usepackage{graphicx}%
\providecommand{\U}[1]{\protect\rule{.1in}{.1in}}
\newtheorem{theorem}{Theorem}[section]
\theoremstyle{plain}

\newtheorem{corollary}[theorem]{Corollary}
\newtheorem{definition}[theorem]{Definition}

\newtheorem{lemma}[theorem]{Lemma}

\newtheorem{proposition}[theorem]{Proposition}

\numberwithin{equation}{section}
\begin{document}
\title[Tetrahedra and Hilbert Spaces ]{Tetrahedra in Hyperbolic Space and Hilbert Spaces with Pick Kernels\medskip}
\author{Richard Rochberg}
\date{}
\maketitle
\tableofcontents

\begin{abstract}
We study of the relation between the geometry of sets in complex hyperbolic
space and Hilbert spaces with complete Pick kernels. We focus on the geometry
associated with assembling sets into larger sets and of assembling Hilbert
spaces into larger spaces. Model questions include describing the possible
triangular faces of a tetrahedron in $\mathbb{CH}^{n}$ and describing the
three dimensional subspaces of four dimensional Hlbert spaces with Pick
kernels. Our novel technical tool is a complex analog of the cosine of a
vertex angle.

\end{abstract}

\section{Introduction and Summary}

We begin with an informal overview, precise statements are in the later sections.

This work is part of a program introduced in \cite{ARSW} and \cite{Ro} of
studying the close relation between the geometry of finite dimensional
reproducing kernel Hilbert spaces (RKHS) with the complete Pick property (CPP)
and the geometry of finite sets in complex hyperbolic space, $\mathbb{CH}%
^{n},$ and real hyperbolic space, $\mathbb{RH}^{n}.$ We focus on assembly
questions, questions about constructing a set or space from designated smaller
parts. The following two questions, which are essentially equivalent, are
specific examples.

\textbf{\noindent Question 1: }Given four triangles in complex hyperbolic
space, $\mathbb{CH}^{n},$ is there a tetrahedron in hyperbolic space with
faces congruent to those triangles?

In Euclidean space of any dimension the necessary and sufficient condition for
three positive numbers to be the side lengths of a triangle is that they
satisfy the triangle inequality and if that holds then those lengths determine
the triangle up to congruence. On the other hand, given four Euclidean
triangles with matching side lengths there might not be a tetrahedron with
faces congruent to those triangles, or, what is the same thing, edges the same
lengths as the sides of those triangles. For instance the numbers $\left\{
4,4,4,4,4,7\right\}  $ are not the edge lengths of a tetrahedron. A necessary
and sufficient condition for there to be a tetrahedron is the nonnegativity of
the determinant of the associated Cayley-Menger matrix, a matrix with entries
constructed using the side lengths. This is an elegant condition but it is an
inequality for a sixth degree polynomial in the lengths, \cite{WD}.

The complexity of that answer results from the choice of side lengths as data.
If instead we select a vertex and use the lengths of sides meeting at that
vertex and the angular structure at the vertex, either vertex angles or
dihedral angles, as data then the necessary and sufficient conditions are
simpler and also hold for tetrahedra in real hyperbolic space. They are
presented in Section \ref{43}.

In real or complex hyperbolic space the necessary and sufficient condition for
three lengths to be the side lengths of a triangle is that they satisfy the
strong triangle inequality, STI, (\ref{STI}) \cite{Ro}. In $\mathbb{CH}^{1}$
and in $\mathbb{RH}^{n}$, $n\geq1,$ those lengths determine the congruence
class of the the triangle, but that is not true in $\mathbb{CH}^{n}$, $n>1.$
However it was shown by Brehm \cite{Br} that in $\mathbb{CH}^{n}$ the three
side lengths together with a fourth number, for instance the angular
invariant, suffice to specify the congruence class of a triangle, and he gives
the explicit conditions those four numbers must satisfy, we recall the details
in Theorem \ref{2}. In Theorem \ref{reduction} we see that the congruence
class of a tetrahedron in $\mathbb{CH}^{n}$ is determined by nine numbers;
Question 1 asks for the constraints those numbers must satisfy.

Questions about point sets in $\mathbb{CH}^{n}$ are equivalent to questions
about Hilbert spaces with reproducing kernels which have the complete Pick
property. In Section \ref{equivalence} we show that Question 1 is equivalent
to the following: .

\textbf{\noindent Question 2: }Given four three dimensional RKHS with the CPP,
is there a four dimensional space with the CPP whose regular three dimensional
subspaces are rescalings of the given spaces$?$

This work began as an effort to better understand an example due to Quiggin
which shows that natural necessary conditions in Question 2 are not
sufficient. That example is analyzed in Section \ref{quiggen}.

We will describe sets in $\mathbb{CH}^{n}$ using two functionals which are
invariant under hyperbolic isometries, the pseudohyperbolic distance between
pairs of points, $\delta(a,b),$ and $\operatorname*{kos}_{a}(b,c),$ a
functional of triples of points which is a complex analog of the cosine of the
vertex angle at $a.$ The "same" functionals are also defined for tuples of
kernel functions in any RKHS and are invariant under rescaling of the space.
In a RKHS $\delta(a,b)$ is related to the angle between kernel functions and
$\operatorname*{kos}_{a}(b,c)$ describes the geometry of the projection of one
kernel function onto the linear span two others. Accepting the guidance of
Klein's Erlangen Program, these invariant quantities should be regarded as
geometric descriptors of sets in $\mathbb{CH}^{n}$ and of RKHS. In Theorem
\ref{general} we use those functionals to give geometric descriptions of
finite sets in $\mathbb{CH}^{n}$ and of a class of Hilbert spaces. Using those
descriptions we build associated moduli spaces and to study assembly questions.

Here is an overview of the contents. In the next section we give background
information about hyperbolic geometry and about Hilbert spaces with
reproducing kernels; in the section after that we recall results connecting
those topics. In Section 4 we discuss various geometric roles of the
functional $\operatorname*{kos}.$ In Section 5 we describe finite sets $X$ in
$\mathbb{CH}^{n},$ and also an associated class of Hilbert spaces, using a
version of spherical coordinates with values of $\operatorname*{kos}$
substituting for cosines of angles. That description involves the positive
semidefiniteness of a matrix $A,$ $A\succcurlyeq0,$ which has the form $A=$
$\left(  \operatorname*{kos}_{1}(i,j)\right)  .$ Submatrices of $A$ encode the
geometry of subsets of $A$ and Sylvester's criterion lets us recast the fact
that $A\succcurlyeq0$ as statements about those submatrices. We use those
facts to relate the geometry of $X$ to the geometry of its subsets. In Section
6 we specialize the results from Section 5 to four point sets and four
dimensional spaces and answer Questions 1 and 2. In both cases a crucial part
of the answer is a condition of the form $\det A\geq0.$ \ If our set $X$ is
inside a copy of $\mathbb{RH}^{k}$ inside $\mathbb{CH}^{n}$ then our results
specialize as results about sets in real hyperbolic space. There is then a
fundamental simplification, the value of $\operatorname*{kos}$ at vertex
simplifies to the cosine of the vertex angle. We develop that theme in detail
for four point sets in Section \ref{rhk} where we give conditions on the
vertex angles and on the dihedral angles at a vertex of a real hyperbolic
tetrahedron. The brief final section contains a few remarks.

\section{Background}

\subsection{Hyperbolic Geometry\label{hg}}

Our background reference for complex hyperbolic geometry is \cite{Go}. We will
use the ball model of complex hyperbolic space $\mathbb{CH}^{n}.$ In that
model the manifold for $\mathbb{CH}^{n}$ is the unit ball, $\mathbb{B}%
_{n}\subset\mathbb{C}^{n},$ and the geometry is determined by the transitive
group of biholomorphic automorphisms of the ball, $\operatorname*{Aut}%
\mathbb{B}_{n}$. \ For each $\alpha\in\mathbb{B}_{n}$ there is a unique
involution $\phi_{a}\in\operatorname*{Aut}\mathbb{B}_{n}$ which satisfies
$\phi_{a}(a)=0.$ The group $\operatorname*{Aut}\mathbb{B}_{n}$ is generated by
those involutions together with the unitary maps. We will say two sets $Z,$
$W\subset$ $\mathbb{CH}^{n}$ are \textit{congruent}, $Z\sim W,$ if there is a
$\phi\in\operatorname*{Aut}\mathbb{B}_{n}$ with $\phi(Z)=W.$ If the sets are
ordered then, absent other comment, we suppose that $\phi$ respects the
ordering. Congruence is an equivalence relation and we are particularly
interested in congruence equivalence classes.

The pseudohyperbolic metric $\delta$ on $\mathbb{CH}^{n}$ is defined by
$\forall\alpha,\beta\in$ $\mathbb{CH}^{n}$
\begin{equation}
\delta(\alpha,\beta)=\left\vert \phi_{\alpha}(\beta)\right\vert =\left\vert
\phi_{\beta}(\alpha)\right\vert . \label{delta}%
\end{equation}
For $n=1$ the formula is $\delta(\alpha,\beta)=\left\vert \frac{\alpha-\beta
}{1-\bar{\beta}\alpha}\right\vert .$ It satisfies the \textit{strong triangle
inequality}, STI; for $x,y,z\in\mathbb{B}_{n}$
\begin{equation}
\frac{\left\vert \delta(x,z)-\delta(z,y)\right\vert }{1-\delta(x,z)\delta
(z,y)}\leq\delta(x,y)\leq\frac{\delta(x,z)+\delta(z,y)}{1+\delta
(x,z)\delta(z,y)}. \label{STI}%
\end{equation}
Equivalently, $\delta$ is the distance on $\mathbb{B}_{n}$ which satisfies
$\delta(0,z)=\left\vert z\right\vert $ for $z\in\mathbb{B}_{n}$ and is
$\operatorname*{Aut}\mathbb{B}_{n}$ invariant \cite{DW}. It is not a length
metric; the length metric it generates is the Bergman metric for the ball
normalized to agree with the Euclidean metric to second order at the origin.

The space $\mathbb{CH}^{n}$ contains various totally geodesic submanifolds of
interest here. These include the classical geodesics which are totally
geodesic copies of $\mathbb{RH}^{1},$ and also includes totally geodesic
copies of $\mathbb{CH}^{1},$ sometimes called \textit{complex geodesics.
}Every pair of points is contained in a unique classical geodesic which is in
turn contained in a unique complex geodesic. There are also totally geodesic
copies of the real hyperbolic plane $\mathbb{RH}^{2}$ inside $\mathbb{CH}%
^{n}.$ In particular set $\mathbb{R}^{2}=\left\{  \left(  x,y,0,...,0\right)
:x,y\in\mathbb{R}\right\}  \subset$ $\mathbb{C}^{n}$ and $BK_{2}%
=\mathbb{R}^{2}\cap\mathbb{CH}^{n}.$ That intersection is a totally
geodesically embedded submanifold whose induced geometry is that of the
Beltrami-Klein model of $\mathbb{RH}^{2}$ with constant curvature $-1/4.$ The
geometry of $BK_{2}$ is not conformal with the Euclidean geometry of the
containing $\mathbb{R}^{2}.$ However the two geometries are conformally
equivalent at the origin, in particular angles with vertex at the origin have
the same size in both geometries. There is a useful picture of $BK_{2}$ in
\cite[pg. 83]{Go} and there is a discussion of the geometry of that model
(although the version with curvature $-1)$ as well as the more familiar
Poincar\'{e} disk model in Appendices C and B of \cite{J}. The automorphic
images of $BK_{2}$ are also totally geodesically embedded submanifolds of real
dimension two, and they, together with the complex geodesics, are the only such.

Similar statements hold for totally geodesic submanifolds of higher dimension.
In particular there are higher dimensional analogs of $BK_{2}$ and at the
origin the induced geometry on $BK_{r}$ is conformal with the Euclidean
geometry of the containing $\mathbb{R}^{r}.$ More information on these topics
is in \cite{Go} and \cite{BI}.

If $G$ is a complex geodesic and $x\in\mathbb{CH}^{n}$ we define the metric
projection of $x$ onto $G,$ $P_{G}x,$ to be that point in $G$ which is closest
to $x$ in the pseudohyperbolic metric.

For $z,w\in\mathbb{B}_{r}$ we define the \textit{kernel function} $k$ by
\begin{equation}
k_{z}(w)=k(w,z)=\frac{1}{1-\left\langle \left\langle w,z\right\rangle
\right\rangle }. \label{kernel}%
\end{equation}
(We write $\left\langle \left\langle \cdot,\cdot\right\rangle \right\rangle $
for the inner product on $\mathbb{C}^{n}$ to distinguish from the inner
products on the general Hilbert spaces we consider.)

There is a fundamental identity which describes the interaction of the
involutive automorphisms with the kernel functions \cite{Ru}: for any
$y,z,w\in\mathbb{B}_{r}$%

\begin{equation}
k(\phi_{y}(z),\phi_{y}(w))=\frac{k\left(  y,y\right)  ^{1/2}}{k(z,y)}%
\frac{k\left(  y,y\right)  ^{1/2}}{k(y,w)}k(z,w). \label{basic}%
\end{equation}

By a triangle in $\mathbb{CH}^{n}$ we mean an ordered set of three distinct
points called vertices, or those vertices together with the sides, the
geodesic segments connecting the vertices. The length of a side is the
$\delta$ distance between the corresponding vertices. Similarly for
tetrahedra. We do not require any nondegeneracy on the sets, for instance
three points could be in a single geodesic.

\subsection{Hilbert Spaces with Reproducing Kernels\label{HSRK}}

Our background references for Hilbert spaces are \cite{AM}, \cite{Ro},
\cite{ARSW2}. Except for the spaces $DA_{r}$ introduced later in this section,
all the Hilbert spaces in this paper are finite dimensional.

An $n$ - dimensional reproducing kernel Hilbert space, RKHS, is an $n$ -
dimensional Hilbert space $H$ together with a distinguished basis of vectors
called \textit{reproducing kernels, }$RK(H)=\left\{  h_{i}\right\}  _{i=1}%
^{n}. $ For any $v\in H$ we write $\hat{v}$ for its normalized version;
$\hat{v}=v/\left\Vert v\right\Vert .$ For $h_{i},h_{j}\in RK(H)$ we write
$h_{ij}$ for their inner product and $\widehat{h_{ij}}$ for the inner product
of their normalizations.%
\[
h_{ij}=\left\langle h_{i},h_{j}\right\rangle ,\ \ \ \widehat{h_{ij}%
}=\left\langle \widehat{h_{i}},\widehat{h_{j}}\right\rangle =\left\langle
\frac{h_{i}}{\left\Vert h_{i}\right\Vert },\frac{h_{j}}{\left\Vert
h_{j}\right\Vert }\right\rangle .
\]
The Gram matrix of $H$ is the matrix $\operatorname*{Gr}(H)=\left(
h_{ij}\right)  _{i,j=1}^{n}.$

A \textit{regular subspace} of $H$ is a subspace $J$ spanned by a subset $S$
of $RK(H)$ and regarded as a RKHS by setting $RK(J)=S.$

We now recall the Drury-Arveson spaces, $DA_{r}$ (sometimes denoted $H_{r}%
^{2}$ and sometimes called the Hardy spaces$),$ some basic references are
\cite{Ar}, \cite{AM}, and \cite{Sh}. With $k_{z}(\cdot)$ the functions from
(\ref{kernel}), $DA_{r}$ is the infinite dimensional RKHS with kernel
functions $\left\{  k_{z}:z\in\mathbb{B}_{r}=\mathbb{CH}^{r}\right\}  $ and
inner product given by $\left\langle k_{s},k_{t}\right\rangle =k_{s}(t).$ We
are particularly interested in finite dimensional regular subspaces of
$DA_{r}$. For $Z=\left\{  z_{j}\right\}  _{j=1}^{n}\subset$ $\mathbb{CH}^{r}$
let $DA_{r}(Z)$ be the regular subspace of $DA_{r}$ spanned by the kernel
functions $\left\{  k_{z_{i}}\right\}  _{i=1}^{n}.$ We abbreviate them by
$\left\{  k_{i}\right\}  $ and set $k_{ij}=\left\langle k_{i},k_{j}%
\right\rangle .$

If $r^{\prime}>r$ there are natural inclusions of $\mathbb{C}^{r}$,
$\mathbb{B}_{r}$ and $\mathbb{CH}^{r}$ into the corresponding objects with
$r^{\prime}.$ These inclusions interact in harmless ways with the
constructions we have discussed and will discuss. For instance, given
$Z\subset\mathbb{B}_{r},$ the natural inclusion of $\mathbb{B}_{r}$ into
$\mathbb{B}_{r^{\prime}}$ takes $Z$ to a set $Z^{\prime}\subset\mathbb{B}%
_{r^{\prime}}.$ There is then an obvious natural map between $DA_{r}(Z)$ and
$DA_{r^{\prime}}(Z^{\prime})$ which preserves all the structure of interest
here. Going forward we will identify such pairs of sets and of spaces and drop
the subscripts on $DA_{r}$ and $DA_{r}(Z).$ In particular when we consider a
finite set $X$ in some $\mathbb{CH}^{n}$ we will generally suppose $n$ is
sufficiently large and not be more specific.

\subsubsection{Rescaling\label{rescaling}}

Given finite dimensional RKHS, $G,H$ with $RK(H)=\left\{  h_{\alpha}\right\}
_{\alpha\in A},$ $RK(G)=\left\{  g_{\beta}\right\}  _{\beta\in B}$ we say $G $
is a \textit{rescaling }of $H$ if there is a one to one map $\theta$ of $A $
onto $B$ and a nonvanishing complex valued function $\gamma$ defined on $A $
such that for all $\alpha_{1},\alpha_{2}$ in $A.$
\begin{equation}
\left\langle h_{\alpha_{1}},h_{\alpha_{2}}\right\rangle =\left\langle
\gamma(\alpha_{1})g_{\theta(\alpha_{1})},\gamma(\alpha_{2})g_{\theta
(\alpha_{2})}\right\rangle =\gamma(\alpha_{1})\overline{\gamma(\alpha_{2}%
)}\left\langle g_{\theta(\alpha_{1})},g_{\theta(\alpha_{2})}\right\rangle
\end{equation}

\noindent If $A$ and $B$ are ordered then, unless we specify otherwise, we
suppose that $\theta$ is monotone increasing. If $G_{H}$ is a regular subspace
of $G$ and $H\sim G_{H}$ we will write $H\leadsto G$ or $H\leadsto
G_{H}\subset G$ and will say that we have a \textit{rescaling of }%
$H$\textit{\ into }$G$\textit{\ }with image\textit{\ }$G_{H}.$

Rescaling is an equivalence relation between spaces $H$ and $G;$ we denote it
by $H\sim G.$ We are especially interested in rescaling equivalence classes.

\subsubsection{Assuming Irreducibility\textbf{\ }}

In \cite[pg. 79]{AM} a RKHS $H$ is called \textit{irreducible }if no two
elements of $RK(H)$ are parallel and no two are orthogonal. Irreducibility is
preserved by rescaling and is inherited by regular subspaces.

With our definition of finite dimensional RKHS the first condition is
automatic. The second is equivalent to requiring that neither $H$ not any of
its regular subspaces can be written as a nontrivial orthogonal direct sum of
two RKHS. It is also equivalent to the requirement that no entry of
$\operatorname*{Gr}(H)$ is zero.

All the spaces we consider will be assumed irreducible and we abbreviate that
class by RKHSI. For those spaces the definitions of invariants in this section
proceed without exceptions.

If $H$ is a finite dimensional space in RKHSI we will write $H\in\mathcal{RK}$.

\subsubsection{The CPP}

The CPP is a property which some $H\in\mathcal{RK}$ have and some do not.
General information about it is in \cite{AM}, \cite{Sh}, and \cite{ARSW2}. We
will see in Theorem \ref{reduction} that the CPP characterizes the class of
Hilbert spaces for which Question 2 can be studied by focusing on Question 1,
and for our purposes here we could have used the statements of that theorem as
our definition of the CPP. However we briefly recall the classical definition.

A multiplier operator on $H\in\mathcal{RK}$ is a linear operator $M$ whose
adjoint $M^{\ast}$ is diagonalized by the kernel functions of $H$. Passing
from a regular subspace $J$ of $H$ to the larger $H$ corresponds to an
enlargement of the set of kernel functions. Hence the adjoint $M_{J}^{\ast}$
of any multiplier $M_{J}$ on $J$ can easily be extended to an operator
$M_{\mathtt{ext}}^{\ast}$ on $H$ which is diagonalized by the kernel functions
of $H.$ By virtue of that structure $M_{\mathtt{ext}}^{\ast}$ will be the
adjoint of a multiplier operator $M_{\mathtt{ext}}$ on $H,$ and it is then
immediate that $\left\Vert M_{\mathtt{ext}}\right\Vert \geq\left\Vert
M_{J}\right\Vert $. If $H\in\mathcal{RK}$ and for any regular subspace $J$ of
$H$ and any multiplier $M_{J}$ on $J$ we can select $M_{\mathtt{ext}}^{\ast}$
so that $\left\Vert M_{\mathtt{ext}}\right\Vert =\left\Vert M_{J}\right\Vert $
then $H$ is said to have the Pick property. If additionally it is true that
matrices of multiplier operators acting on spaces of tuples of elements of $J$
always have such an extension with equal norm then $H$ is said to have the
complete Pick property, CPP.

If $H\in\mathcal{RK}$ has the CPP then we write $H\in\mathcal{CPP}$.

If $H\in\mathcal{CPP}$ then it is straightforward that any regular subspace of
$H$ also has the CPP. Also it is not hard to see that if $H\sim J$ then
$J\in\mathcal{CPP}$. We will use both facts without mention.

It is fundamental that the spaces $DA$ all have the CPP and hence so do the
finite dimensional regular subspaces $DA(Z).$ In the other direction we see in
Theorem \ref{reduction} that if $H\in\mathcal{CPP}$ then there is a
$Z\subset\mathbb{CH}^{n}$ with $H\sim DA(Z)$. Furthermore, for finite $X,W$ we
have $X\sim W$ if and only if $DA(X)\sim DA(W)$; one direction follows from
earlier remarks and both directions are contained in Theorem \ref{reduction}.

\subsection{Invariants\label{define}}

In this section we suppose $H\in\mathcal{RK}$ and introduce several functions
of tuples of elements from $RK(H).$ The values of the functionals are
determined by entries of the Gram matrix of $H$ and hence are unchanged if $H$
is a regular subspace of some $H^{\prime}\in\mathcal{RK}$ and we regarded the
functionals as acting on elements of $RK(H^{\prime}).$ Also, it will be clear
from their formulas that these particular functionals are invariant under
rescaling of $H.${}

Such functionals can also be regarded as functionals of tuples of points in
$\mathbb{CH}^{n}$ using the following scheme. If, for instance, $F$ is defined
on pairs of kernel functions then for $x,y\in\mathbb{CH}^{n}$ we define $F$
acting on $x,y$ by selecting a finite $X$ containing both points, regarding
$F$ as acting on $RK(DA(X))$ and setting $F(x,y)=F(k_{x},k_{y}).$ By the
previous remarks this definition does not depend on the choice of $X$ and we
will not mention $X$ again. (We introduced $X$ only to avoid using the
infinite dimensional space $DA.)$ \ We noted that automorphisms of
$\mathbb{CH}^{n}$ induce rescalings of the $DA(X)$ and that the functionals we
will consider are rescaling invariant. Hence the functionals, regarded as
acting on tuples of points in $\mathbb{CH}^{n},$ are invariant under
automorphisms; they depend only on the congruence class of the tuple of
arguments. Going forward we will use the same names and notation for the
functionals acting on a $RK(H)$ and for the induced functionals acting on
points in $\mathbb{CH}^{n}.$

$\mathbf{\delta:}$ For $H\in\mathcal{RK}$ and $h_{1},$ $h_{2}\in RK(H)$ we
define $\delta_{H}$ by
\begin{equation}
\delta_{H}(i,j)=\delta_{H}(h_{i},h_{j})=\sqrt{1-\frac{\left\vert
h_{ij}\right\vert ^{2}}{h_{ii}h_{jj}}}=\sqrt{1-|\widehat{h_{ij}}|^{2}}.
\label{distance}%
\end{equation}
This function is a metric on $RK(H)$ \cite[Lemma 9.9]{AM}, \cite{ARSW}.
Clearly $\delta_{H}$ is invariant under rescaling.

The metric $\delta_{DA}$ on $\mathbb{B}_{n}$ equals the pseudohyperbolic
metric $\delta$. To see this use the definition of $k$ to rewrite
(\ref{basic}) as, for $w,x,y,z\in\mathbb{CH}^{n},$
\begin{equation}
1-\frac{k(w,y)k(y,z)}{k(y,y)k(w,z)}=\left\langle \left\langle \phi_{y}%
(z),\phi_{y}(w)\right\rangle \right\rangle . \label{xxx}%
\end{equation}
Taking $z=w$ we see%
\begin{equation}
\delta^{2}(y,w)=\left\vert \phi_{y}(w)\right\vert ^{2}=1-\frac{\left\vert
k(y,w)\right\vert ^{2}}{k(y,y)k(w,w)}=\delta_{DA}^{2}(y,w). \label{2.13}%
\end{equation}
The first equality in (\ref{2.13}) is the definition of $\delta$, the second
is the special case of (\ref{xxx}), the third is the definition of
$\delta_{DA}^{2}.$ Alternatively, for any $z\in\mathbb{B}_{n}$ note that
$\delta(0,z)=\left\vert z\right\vert =\delta_{DA}(0,z)$ and both $\delta$ and
$\delta_{DA}$ are invariant under automorphisms; hence the two metrics are equal.

Going forward we will write $\delta$ for the various $\delta_{DA(Z)}$ and for
the pseudohyperbolic metric. Using (\ref{2.13}) we can express $\delta$ in
terms of coordinates; for $y,w\in\mathbb{B}_{n}$
\begin{equation}
\delta^{2}(y,w)=1-\frac{(1-\left\langle \left\langle y,y\right\rangle
\right\rangle )(1-\left\langle \left\langle w,w\right\rangle \right\rangle
)}{\left\vert 1-\left\langle \left\langle y,w\right\rangle \right\rangle
\right\vert ^{2}}. \label{2.1111}%
\end{equation}

$\mathbf{\alpha:}$ For $H\in\mathcal{RK}$ we define the \textit{angular
invariant} $\alpha$ by, for $k_{1},$ $k_{2},$ $k_{3}\in RK(H)$
\begin{equation}
\alpha(1,2,3)=\alpha(k_{1},k_{2},k_{3})=-\arg\left\langle k_{1},k_{2}%
\right\rangle \left\langle k_{2},k_{3}\right\rangle \left\langle k_{3}%
,k_{1}\right\rangle =-\arg k_{12}k_{23}k_{31}. \label{def ang}%
\end{equation}
\textit{(}In general situations care is needed in selecting a branch of
$\arg.$ However for $k\in RK(DA(Z)),$ $\operatorname{Re}k>0$ and that lets us
avoid problems.) Notice that $\alpha$ satisfies a cocycle identity; if $k_{4}$
is a fourth kernel function then
\begin{equation}
\alpha(1,2,3)-\alpha(2,3,4)+\alpha(3,4,1)-\alpha(4,1,2)=0. \label{kkocycle}%
\end{equation}
We discuss the geometry associated with $\alpha$ in Section \ref{area}. More
about this invariant is in \cite{Go}, \cite{CO}, \cite{C}, \cite{M}.

$\mathbf{kos:}$ For $H\in\mathcal{RK}$ we define $\operatorname*{kos}$, a
functional of triples of kernel functions. For $k_{1},$ $k_{2},$ $k_{3}\in
RK(H)$, $k_{1}\neq k_{2},$ $k_{3},$ set
\begin{equation}
\operatorname*{kos}\nolimits_{k_{1}}(k_{2},k_{3})=\operatorname*{kos}%
\nolimits_{1}(2,3)=\frac{1}{\delta_{12}\delta_{13}}\left(  1-\frac
{k_{21}k_{13}}{k_{11}k_{23}}\right)  , \label{def kos}%
\end{equation}
and note the symmetry $\operatorname*{kos}\nolimits_{1}(2,3)=\overline
{\operatorname*{kos}\nolimits_{1}(3,2)}.\ $

If $H=DA(X)$ for some $X\subset\mathbb{CH}^{n}$ then $\operatorname*{kos}$ is
related to the geometry of $X.$ Recall that for $x\in\mathbb{CH}^{n}$
$\phi_{x}$ is the ball involution which interchanges $x$ and $0.$ We can use
(\ref{xxx}) and (\ref{2.13}) to obtain
\begin{equation}
\operatorname*{kos}\nolimits_{x}(y,z)=\frac{1}{\delta(x,y)\delta
(x,z)}\left\langle \left\langle \varphi_{x}\left(  y\right)  ,\varphi
_{x}\left(  z\right)  \right\rangle \right\rangle =\left\langle \left\langle
\widehat{\varphi_{x}\left(  y\right)  },\widehat{\varphi_{x}\left(  z\right)
}\right\rangle \right\rangle . \label{first kos}%
\end{equation}
In particular, if $x_{1}$ is at the origin then $\delta(x_{1},w)=\left\vert
w\right\vert $ and $\phi_{x_{1}}$ is the identity. In that case
\begin{equation}
\operatorname*{kos}\nolimits_{x_{1}}(y,z)=\left\langle \left\langle \frac
{y}{\left\vert y\right\vert },\frac{z}{\left\vert z\right\vert }\right\rangle
\right\rangle =\left\langle \left\langle \widehat{y},\widehat{z}\right\rangle
\right\rangle . \label{simple kos}%
\end{equation}
(Although $\operatorname*{kos}$ is invariant under automorphisms of
$\mathbb{B}_{n} $ this formula is an inhomogenous representation and is not invariant.)

If the vectors $y$ and $z$ were in $\mathbb{R}^{n}$ this would be the inner
product of unit vectors in $\mathbb{R}^{n}$ and hence would equal the cosine
of the angle between the segments $0y$ and $0z$. That is the source of the
name $\operatorname*{kos}.$

If $\dim H=n$ then for $1\leq s\leq n$ we define the $(n-1)\times(n-1)$
matrices%
\begin{align}
\operatorname*{KOS}(H,s)  &  =\operatorname*{KOS}(\operatorname*{Gr}%
(H),s)=\left(  \operatorname*{kos}\nolimits_{s}(i,j)\right)  _{\substack{i,j=1
\\i,j\neq s}}^{n},\label{KOS}\\
\operatorname*{MQ}(H,s)  &  =\left(  \delta_{si}\delta_{sj}\operatorname*{kos}%
\nolimits_{s}(i,j)\right)  _{\substack{i,j=1 \\i,j\neq s}}^{n}=\left(
1-\frac{k_{is}k_{sj}}{k_{ij}k_{ss}}\right)  _{\substack{i,j=1 \\i,j\neq
s}}^{n} \label{MQx}%
\end{align}
We also write $\operatorname*{KOS}(X,s)$ for $\operatorname*{KOS}(DA(X),s).$

$\qquad\mathbf{X:}$\textbf{\ }In \cite{KR} Kor\'{a}nyi and Reimann introduced
a functional of $4$-tuples of points in $\partial\mathbb{B}_{n},$ the ideal
boundary of $\mathbb{CH}^{n}.$ Their definition extends in a natural way to an
automorphism invariant functional on $4$-tuples of points in $\mathbb{CH}%
^{n},$ and more generally to a rescaling invariant functional of $4$-tuples of
kernel functions in a RKHSI.

\begin{definition}
Suppose $H\in\mathcal{RK}$. Given $\left\{  k_{i}\right\}  _{i=1}^{4}\subset
RK(H)$ the Kor\'{a}nyi-Reimann cross ratio, $\mathbf{X}(k_{1},k_{2}%
,k_{3},k_{4}),$ is defined as%
\[
\mathbf{X}(k_{1},k_{2},k_{3},k_{4})=\frac{k_{31}k_{42}}{k_{32}k_{14}}.
\]

\end{definition}

Although $\mathbf{X}$ is often used when describing and studying the geometry
of sets in $\mathbb{CH}^{n}$, for instance \cite{KR}, \cite{Go}, \cite{CG},
and the invariants we discussed earlier in this section can be defined using
$\mathbf{X,}$ we will not use $\mathbf{X}$ below.

$\operatorname*{Gr}\mathbf{(H):}$ Having discussed several invariants we
should emphasize that for $H\in\mathcal{RK}$ the Gram matrix,
$\operatorname*{Gr}(H)$, is not invariant under rescaling of $H.$ However
$\operatorname*{Gr}(H)$ is stable under passage from $H$ to a regular subspace
$J;$ $\operatorname*{Gr}(J)$ is the submatrix of $\operatorname*{Gr}(H)$
obtained by retaining all and only the rows and columns of $\operatorname*{Gr}%
(H)$ which are built using reproducing kernels of $H$ which are also in $J.$

\subsection{Matrix Notation\label{notation}}

For an $n\times n$ matrix $A$ we write $A\succ0$ if $A$ is positive definite
and $A\succcurlyeq0$ if it is positive semidefinite. We say $B$ is a
\textit{principal submatrix} of $A$ if it is obtained from $A$ by removing
certain rows and also the corresponding columns. Note that if $H\in$
$\mathcal{RK}$ then $J$ is a regular subspace of $H$ if and only
$\operatorname*{Gr}(J)$ is a principal submatrix of $\operatorname*{Gr}(H)$.
We denote the set of all principal submatrices of $A$ by $\mathcal{PS}%
{\large (A).}$ If $B\in\mathcal{PS}{\large (A)}$ and the rows and columns
retained in $B$ are, for some specified $k\leq n,$ those with indices $j,$
$1\leq j\leq k$ then $B$ is said to be a \textit{leading} principal submatrix.
The determinants of those matrices are called \textit{principal minors} and
\textit{leading principal minors }respectively.

\section{The $\mathcal{CPP}$ and Point Sets in $\mathbb{CH}^{n}$}

We will often use a model triangle $\Gamma$ or a model tetrahedron $\Delta$
which have convenient coordinates. Our model triangle is $\Gamma
\subset\mathbb{CH}^{2}$: \
\begin{align}
&  \Gamma=\left\{  x_{1},x_{2},x_{3}\right\}  =\left\{
(0,0\},(a,0),(x,b)\right\}  ,\label{tri}\\
&  a>0,b\geq0\text{, \ }x\in\mathbb{C},\nonumber\\
&  0<a,\text{ }\left\vert x\right\vert ^{2}+b^{2}<1.\nonumber
\end{align}
Our model tetrahedron is $\Delta\subset\mathbb{CH}^{3}$:
\begin{align}
&  \Delta=\left\{  x_{1},x_{2},x_{3},x_{4}\right\}  =\left\{
(0,0,0),(a,0,0),(x,b,0),(y,z,c)\right\}  ,\label{tetra}\\
&  a>0,\text{ }b,c\geq0,\text{ }x,y,z\in\mathbb{C}\text{, \ }\nonumber\\
&  \left\vert a\right\vert ^{2},\text{ }\left\vert x\right\vert ^{2}%
+\left\vert b\right\vert ^{2},\text{ }\left\vert y\right\vert ^{2}+\left\vert
z\right\vert ^{2}+\left\vert c\right\vert ^{2}<1\nonumber
\end{align}

To a three dimensional $H\in\mathcal{RK}$ with $RK(H)=\left\{  h_{i}\right\}
_{i=1}^{3}$ we associate the following data sets:
\begin{align*}
S  &  =\left\{  |\widehat{h_{12}}|,\text{ }|\widehat{h_{23}}|\text{,
}|\widehat{h_{13}}|\text{, }\mathbb{\alpha}_{123}\right\}  ,\\
S^{\prime}  &  =\left\{  \delta_{12},\text{ }\delta_{13},\text{ }\delta
_{23}\text{, }\mathbb{\alpha}_{123}\right\}  ,\\
S^{\prime\prime}  &  =\{\delta_{12},\delta_{13},\operatorname*{kos}%
\nolimits_{1}(2,3)\}.
\end{align*}
And for convenience we set%
\begin{equation}
\Gamma_{abc}=\left\vert \frac{\widehat{h_{ab}}\widehat{h_{bc}}}%
{\widehat{h_{ca}}}\right\vert =\sqrt{\frac{(1-\delta_{H}^{2}(a,b))\,(1-\delta
_{H}^{2}(b,c))}{(1-\delta_{H}^{2}(c,a))}.} \label{ij}%
\end{equation}

\noindent All these quantities are unchanged by rescaling $H.$ Also note that
the $\Gamma^{\prime}s$ can be computed using entries from $\operatorname*{Gr}%
(H),$ or directly from the data in $S$ or in $S^{\prime}.$

The following describes three point sets $X$ in $\mathbb{CH}^{2}$ and the
associated $DA(X)$ spaces.

\begin{theorem}
[\cite{Br} \cite{AM} \cite{Ro}]\label{2}\label{triangle} Given a three
dimensional $H\in\mathcal{RK}$ the following are equivalent:

\begin{enumerate}
\item $H\in\mathcal{CPP}$.

\item There is a three point set $X$ in $\mathbb{CH}^{2}$ with $H\sim DA(X).$

\item There is a $\Gamma$ as in (\ref{tri}) with $H\sim DA(\Gamma).$

\item Let $J$ be the regular subspace of $H$ spanned by $\{h_{1},$ $h_{2}\}.$
Let $M$ be the multiplier on $J$ of norm one specified by the action of its
adjoint,
\begin{align*}
M^{\ast}h_{1}  &  =0,\\
M^{\ast}h_{2}  &  =\delta_{H}(h_{1},h_{2})h_{2};
\end{align*}
then $M$ extends to a multiplier of norm one on $H.$

\item $\operatorname*{KOS}(H,1)\succcurlyeq0$.

\item
\begin{equation}
_{.\text{ }}\left\vert \operatorname*{kos}\nolimits_{1}(2,3)\right\vert \leq1.
\label{kos estimate}%
\end{equation}

\item $S$ and the $\Gamma$'s defined from $S$ using (\ref{ij}) satisfy%
\begin{equation}
\Gamma_{123}+\Gamma_{231}+\Gamma_{312}\leq2\cos\mathbb{\alpha}_{123}.
\label{aaa}%
\end{equation}

\item $S^{\prime}$ and the $\Gamma$'s defined from $S^{\prime}$ using
(\ref{ij}) satisfy (\ref{aaa}).
\end{enumerate}

Furthermore $X$ sits inside a complex geodesic if and only if $\det
\operatorname*{KOS}(H,1)=0,$ or equivalently $\left\vert \operatorname*{kos}%
\nolimits_{1}(2,3)\right\vert =1,$ or the $b$ coordinate of $\Gamma$ in (3)
equals $0.$

Conversely, given

\qquad data $S$ and $\Gamma$'s defined from $S$ using (\ref{ij}) such that
(\ref{aaa}) holds,

or

\qquad data $S^{\prime}$ and $\Gamma$'s defined from $S^{\prime}$ using
(\ref{ij}) such that (\ref{aaa}) holds,

or

\qquad data $S^{\prime\prime}$ $f$or which (\ref{kos estimate}) holds,

there is triangle $X$ in $\mathbb{CH}^{2},$ unique up to congruence, which has
those parameters.
\end{theorem}

\begin{proof}
All of this is in the references mentioned except for the statement about $X$
being in a complex geodesic. That fact is implicit in the proof of Theorem 16
in \cite{Ro}.
\end{proof}

Thus each of $S,$ $S^{\prime},$ or $S^{\prime\prime}$ could be used to
describe a triangle. The equivalence of using $S$ or $S^{\prime}$ is clear.
Passing between $S^{\prime}$ and $S^{\prime\prime}$ is described in Section
\ref{invariants}.

Some aspects of the previous theorem extend to larger sets and spaces. The
next result is an amalgam of the fact that up to rescaling the $H\in
\mathcal{CPP}$ are exactly the spaces $DA(X)\ $for $X$ a finite set in
$\mathbb{CH}^{n}$ \cite{AM}, the fact mentioned earlier that automorphisms of
the ball induce rescalings of spaces\ $DA(X),$ and the description of
congruence classes of finite sets in $\mathbb{CH}^{n}$ given in \cite{BE},
\cite{G}, \cite{HS}, and \cite{Ro}.

\begin{theorem}
\label{reduction}An $n$ dimensional $H\in\mathcal{RK}$ satisfies
$H\in\mathcal{CPP}$ if and only if there is an $X=\left\{  x_{i}\right\}
_{1}^{n}\subset\mathbb{CH}^{n-1}\ with$ $H\sim DA(X).$

$Given$ an $n$ dimensional $H^{\prime}\in\mathcal{CPP}$ with $H^{\prime}\sim
DA(X^{\prime})$ for $X^{\prime}=\left\{  x_{i}^{\prime}\right\}  _{1}^{n}$ the
following are equivalent:

\begin{enumerate}
\item $H$ $\sim H^{\prime}.$

\item $X\sim X^{\prime}.$

\item DA($X)\sim DA(X^{\prime}).$

\item All the triangles of $X$ are congruent to the triangles of $X^{\prime}
$; i.e. for $1\leq$ $i,$ $j,$ $k\leq n$ there is a ball automorphisms taking
$\{x_{i},$ $x_{j},$ $x_{k}\}$ to $\{x_{i}^{\prime},$ $x_{j}^{\prime},$
$x_{k}^{\prime}\}.$

\item The triangles of $X$ which have one vertex at $x_{1}$ are congruent to
the corresponding triangles of $X^{\prime}$; for $1<i<j\leq n$ the triangles
$\{x_{1},$ $x_{i},$ $x_{j}\}$ and $\{x_{1}^{\prime},$ $x_{i}^{\prime},$
$x_{j}^{`\prime}\}$ are congruent.

\item The regular three dimensional subspaces of $H$ are rescalings of the
corresponding regular three dimensional subspaces of $H^{\prime}.$

\item The regular three dimensional subspaces of $H$ which contain $k_{x_{i}}
$ are rescalings of the corresponding three dimensional subspaces of
$H^{\prime}.$
\end{enumerate}
\end{theorem}

A consequence is

\begin{corollary}
\label{allthesame}There are bijections between the class of $n$ dimensional
$H\in\mathcal{CPP}$ modulo rescaling, the class of $n$ dimensional
$\mathcal{RK}$ of the form $DA(X)$ modulo rescaling, and the class of $n$
point sets $X $ in $\mathbb{CH}^{k}$ modulo congruence. The correspondences
respect inclusions between spaces and between sets.
\end{corollary}

If we use $S,$ $S^{\prime},$ or $S^{\prime\prime}$ to characterize the
triangles in (4) then we obtain $O(n^{3})$ real numbers which describe $X$ up
to congruence. However that list has repetitions and redundancies. Restricting
to the triangles listed in (5) and adjusting for the fact that some side
lengths are listed twice produces a list of $\left(  n-1\right)  ^{2} $
independent real parameters which determine $X$ up to congruence. That number
is optimal; triangles are determined by 4 parameters, tetrahedra by 9. There
are also constraints on the parameters, the inequality (\ref{aaa}) for
triangles is the simplest example. We give analogous constraints for
tetrahedra in Theorem \ref{Q 1} below. Also, noting the previous corollary,
the same parameters (or, perhaps, similarly named parameters) describe spaces
$DA(X)$ and spaces in $\mathcal{CPP}$.

\subsection{Assembly and Coherence}

It may be that we have several spaces $\left\{  J_{i}\right\}  $ and a
rescaling of each with a $H_{i}$ contained in $H;$ $J_{i}\leadsto H_{i}\subset
H.$ If that holds then there are \textit{coherence conditions }connecting the
$\left\{  J_{i}\right\}  $ with each other and the $\left\{  J_{i}\right\}  $
are said to be a \textit{coherent set} of spaces. Informally the conditions
are that subspaces of various $J_{i}$ which are rescalings of the same
subspace of $H$ must be rescalings of each other. Rather than give a detailed
definition we will describe the situation in detail in the context of Question 2.

For any four dimensional $H\in\mathcal{CPP}$ with $RK(H)=\left\{
h_{i}\right\}  _{i=1}^{4}$ denote the four regular three dimensional subspaces
$\left\{  H_{i}\right\}  _{i=1}^{4}$ by
\[
RK(H_{i})=\left\{  h_{r}:1\leq r\leq4,\text{ }r\neq i\right\}  .
\]
Suppose we have four three dimensional spaces $\left\{  J_{i}\right\}
_{i=1}^{4}\subset\mathcal{RK}$. Question 2 asks for necessary and sufficient
conditions on the $\left\{  J_{i}\right\}  _{i=1}^{4}$ to insure that there is
an $H\in\mathcal{CPP}$ with for $1\leq i\leq4,$ $H_{i}\sim J_{i}.$ If there
are such rescalings then we suppose the indices on the $j_{ir}$ have been
chosen so that for each $i,r$ the rescaling of $H_{i}$ and $J_{i}$ matches the
kernel function $h_{r}$ of $H_{i}$ with the kernel function $j_{ir}$ of
$J_{i}.$ Now notice that, given the existence of the rescalings there must be
relationships between some subspaces of the various $\left\{  J_{i}\right\}
.$ The relationships are all of the same form, we will describe one particular
case. The space $J_{1},$ with kernel functions $\left\{  j_{12},j_{13}%
,j_{14}\right\}  ,$ is a rescaling of $H_{1}$ which has kernel functions
$\left\{  h_{2},h_{3},h_{4}\right\}  $ with $j_{1s}$ pairing with $h_{s};$
similarly for $J_{3}$ and $H_{3}$ with $\left\{  j_{31},j_{32},j_{34}\right\}
$ and $\left\{  h_{1},h_{2},h_{4}\right\}  .$ The subspace $J_{13}$ of $J_{1}$
with kernel functions $\left\{  j_{12},j_{14}\right\}  $ and the subspace
$J_{31}$ of $J_{3}$ with kernel functions $\left\{  j_{32},j_{34}\right\}  $
are both rescalings of the subspace of $H_{24}$ with kernel functions
$\left\{  h_{2},h_{4}\right\}  ,$ and hence $J_{13}\sim J_{31}.$ Define
$J_{pq}$ similarly for $1\leq p\neq q\leq4.$ The coherence conditions on the
$\left\{  J_{i}\right\}  $ are the collection of all relations $J_{pq}\sim
J_{qp}$. They are necessary for the rescalings of $\left\{  J_{i}\right\}  $
into $H.$ If the $\left\{  J_{i}\right\}  $ satisfy these conditions, even if
we do not know that there is an $H,$ we write $\left\{  J_{i}\right\}
\rightrightarrows\,??.$ If furthermore we know there is an $H$ we write
$\left\{  J_{i}\right\}  \rightrightarrows H.$ We use the same notation if
there are more $\left\{  J_{i}\right\}  $ of they have larger dimension. In
general the dimension of the overlap sets may be larger than two. (However for
spaces with the CPP Condition (7) of Theorem \ref{reduction} insures that the
conditions on overlap of two and three dimensional sets imply the coherence
conditions for the larger sets.) Thus the statement $\left\{  J_{i}\right\}
\rightrightarrows\,??$ is the statement that there is a description (perhaps
only implicit) of a type of target space $H$ and an assembly scheme for
constructing such an $H$ from overlapping copies of the $\left\{
J_{i}\right\}  $ and, furthermore, the $\left\{  J_{i}\right\}  $ satisfy the
coherence conditions which are necessary for such an assembly to be possible.

If $\left\{  J_{i}\right\}  \subset\mathcal{CPP}$ then by Corollary
\ref{allthesame} these statements about spaces, subspaces, rescalings, values
of invariants and coherence are equivalent to statements about sets in
$\mathbb{CH}^{m},$ subsets, congruences, values of invariants and an
appropriate notion of coherence, and we will use the same language and
notation in that context. For instance given sets $\left\{  Y_{i}\right\}  $
in $\mathbb{CH}^{n} $ we will write $\left\{  Y_{i}\right\}  \rightrightarrows
X$ if there are congruences of each $Y_{i}$ into some $X$ which satisfy a
preassigned scheme saying which points from the $\left\{  Y_{i}\right\}  $ are
to be mapped to which points of $X.$ We write $\left\{  Y_{i}\right\}
\rightrightarrows\,??$ if we do not know there is an $X$ but do have the
congruences among various subsets $\left\{  Y_{i}\right\}  $ which would be
mapped to the same subsets of $X$. Sometimes we will pass between the
formulations with spaces and with sets.

If we have $\left\{  J_{i}\right\}  \rightrightarrows H$ we would like obtain
information about $H$ from the spaces $\left\{  J_{i}\right\}  $ together with
the coherence data $\left\{  J_{i}\right\}  \rightrightarrows\,??.$ We know
from Theorem \ref{reduction} that the values of $\delta_{ij}$ and
$\operatorname*{kos}_{i}(j,k)$ for $H$ describe $H$ up to rescaling and would
like to compute them from the $\left\{  J_{i}\right\}  .$ Given $\left\{
J_{i}\right\}  \rightrightarrows\,??$ we can construct the imputed value of
$\delta_{H}(a,b)$ by selecting any $J_{s}$ whose image under the rescaling
$J_{s}\leadsto H$ contains the kernel functions $h_{a}$ and $h_{b}$ of $RK(H).
$ We write $j_{ra},$ and $j_{rb}$ for the elements of $RK(J_{s})$ which
correspond to $h_{a}$ and $h_{b}$ under that rescaling map and use
$\delta_{J_{r}}(j_{ra},j_{rb})$ as our constructed value. Note that this is
defined even if there is no $H,$ however if there is an $H$ then the rescaling
$J_{s}\leadsto H$ insures that this value is $\delta_{H}(h_{a},h_{b}).$ Also,
if there is another possible choice for $J_{s}$ then the coherence conditions
insure it will produce the same value. If there is no such $J_{s}$ then we
have no candidate for the value $\delta_{H}(a,b).$ The procedure for
constructing our candidate for $\operatorname*{kos}_{a}(b,c)$ is the same
except that we need to select a $J_{s}$ whose image would contain the three
elements of $RK(H)$ with indices $a,b,$ and $c.$

Using the candidate values of $\operatorname*{kos}_{a}(b,c)$ we construct a
matrix, perhaps only partially defined, which we denote $\operatorname*{KOS}%
(\left\{  J_{i}\right\}  ,a).$ If $J_{i}=DA(X_{i})$ we may write
$\operatorname*{KOS}(\left\{  X_{i}\right\}  ,a)$ for $\operatorname*{KOS}%
(\left\{  J_{i}\right\}  ,a).$ If there is an $H$ with $\left\{
J_{i}\right\}  \rightrightarrows H$ then the defined values in this matrix
will equal those in $\operatorname*{KOS}(H,a).$ However $\operatorname*{KOS}%
(\left\{  J_{i}\right\}  ,a)$ is constructed from $\left\{  J_{i}\right\}  $
and the coherence data, without needing $H.$ Hence comparing properties of
$\operatorname*{KOS}(\left\{  J_{i}\right\}  ,a)$ to properties any
$\operatorname*{KOS}(H,a)$ must have is a test of the possibility there is
such an $H.$

\subsection{Equivalence \textbf{of} the Two Questions\label{equivalence}}

Using the previous theorems we can see that the two questions in the
introduction are equivalent. Those theorems give us the following facts:

\begin{enumerate}
\item Given a set of triangles $\left\{  Y_{i}\right\}  _{i=1}^{4}$ in
$\mathbb{CH}^{n}$ there are three dimensional $\left\{  J_{i}\right\}
_{i=1}^{4}\subset\mathcal{CPP}$ such that
\begin{equation}
J_{i}\sim DA(Y_{i}),\text{ }i=1,...,4, \label{3d}%
\end{equation}
In the other direction, given 3 dimensional $\left\{  J_{i}\right\}
_{i=1}^{4}\subset\mathcal{CPP}$ there are $\left\{  Y_{i}\right\}  _{i=1}^{4}$
such that (\ref{3d}) holds. In either case (\ref{3d}) continues to hold if the
$\left\{  J_{i}\right\}  \ $are replaced by rescalings $\left\{  J_{i}%
^{\prime}\right\}  $ or if the $\left\{  Y_{i}\right\}  $ are replaced by
congruent sets $\left\{  Y_{i}^{\prime}\right\}  $.

\item Given a tetrahedron $X$ in $\mathbb{CH}^{n}$ there is an $H\in
\mathcal{CPP}$ such that
\begin{equation}
H\sim DA(X). \label{4d}%
\end{equation}
In the other direction, given $H\in\mathcal{CPP}$ there is an $X$ such that
(\ref{4d}) holds. In either case (\ref{4d}) continues to hold if $H$ is
replaced by a rescaling $H^{\prime}$ or $X$ is replaced by a congruent
$X^{\prime}.$
\end{enumerate}

\begin{proposition}
\label{equivalent}Given triangles $\left\{  Y_{i}\right\}  _{i=1}^{4}$ in
$\mathbb{CH}^{n}$ and Hilbert spaces $\left\{  J_{i}\right\}  _{i=1}%
^{4}\subset\mathcal{CPP}$ which are related as in (\ref{3d}), there is an $X$
with $\left\{  Y_{i}\right\}  \rightrightarrows X$ if and only if there is an
$H\in\mathcal{CPP}$ with $\left\{  J_{i}\right\}  \rightrightarrows H.$ In
that case $H$ and $X$ satisfy (\ref{4d}).
\end{proposition}

\begin{proof}
This follows from statements (1) and (2) above together with the observation
that for a finite $X\subset\mathbb{CH}^{n}$ the regular subspaces of $DA(X)$
are exactly the spaces $DA(Y)$ for $Y$ a subset of $X.$
\end{proof}

In the proposition the assumption that (\ref{3d}) holds insures the $J_{i}%
\in\mathcal{CPP}$. Even with that condition, having $H\in\mathcal{RK}$ and
$\left\{  J_{i}\right\}  \rightrightarrows H$ is not enough to insure that
$H\in\mathcal{CPP}$. This is shown by, for instance, Quiggin's example which
is discussed in Section \ref{quiggen}.

\section{Geometry and $Kos$\label{gk}}

In this section we describe various relations between values of
$\operatorname*{kos}$ and the geometry of sets in $\mathbb{CH}^{n}.$

\subsection{Evaluating $Kos$\label{EK}}

Because the functional $\operatorname*{kos}$ is an automorphism invariant we
can study it for a general triple by first using an automorphism to place our
triple in the configuration $\Gamma$ described in (\ref{tri}): $\left\{
x_{1},x_{2},x_{3}\right\}  =\left\{  \left(  0,0\right)  ,\left(  a,0\right)
,\left(  x,b\right)  \right\}  $. In that case computing using
(\ref{simple kos}) gives
\begin{equation}
\operatorname*{kos}\nolimits_{1}(2,3)=\left\langle \left\langle \frac
{(a,0)}{\left\Vert (a,0)\right\Vert },\frac{\left(  x,b\right)  }{\left\Vert
\left(  x,b\right)  \right\Vert }\right\rangle \right\rangle =\frac{a\bar{x}%
}{\left\Vert x_{2}\right\Vert \left\Vert x_{3}\right\Vert }=\frac{\bar{x}%
}{\left\Vert x_{3}\right\Vert }. \label{computed kos}%
\end{equation}
An invariant formulation of that statement will let us evaluate
$\operatorname*{kos}\nolimits_{1}(2,3)$ for a general triple. Suppose
$\left\{  x_{1},x_{2},x_{3}\right\}  $ is a given triple and $G(1,2)$ is the
complex geodesic which contains $x_{1}$ and $x_{2}$. Recall that $P_{G(1,2)}$
is the metric projection onto $G(1,2).$ Here is the invariant statement.

\begin{proposition}
\label{formula}Set $P_{G(1,2)}x_{3}=y.$ Writing $\operatorname*{kos}%
\nolimits_{x_{1}}(x_{2},x_{3})=re^{i\theta},$ $r>0,$ $-\pi\leq\theta<\pi$ and
$\operatorname*{angle}$ for the hyperbolic angle we have
\begin{align}
r  &  =\frac{\delta(x_{1},y)}{\delta(x_{1},x_{3})},\label{kos1}\\
\theta &  =\operatorname*{angle}(x_{1}x_{2},x_{1}y). \label{kos2}%
\end{align}

\end{proposition}

\begin{proof}
We need to show the formula is correct for $\Gamma$ and that it is invariant.
For $\Gamma$ the complex geodesic $G(1,2)$ is the unit disk in the first
coordinate line and in that case it is elementary to show that the metric
projection of $x_{2}$ onto $G(1,2)$ is $(x,0)$. Also, for angles with vertex
at the origin the Euclidean angle and hyperbolic angle are the same. That is
enough to establish that both (\ref{kos1}) and (\ref{kos2}) are correct for
$\Gamma.$

To see that (\ref{kos1}) is invariant note that the statement $P_{G(1,2)}%
x_{3}=y$ is invariant as is $\delta.$ The equality (\ref{kos2}) is more subtle
because there is no natural definition of the angle of intersection for two
geodesics in $\mathbb{CH}^{n}.$ However the geodesic segments $x_{1}%
P_{G(1,2)}x_{3}$ and $x_{1}x_{2}$ are both in $G(1,2),$ and any complex
geodesic is conformally equivalent to the classical Poincar\'{e} disk which
does carry an invariant notion of angle between intersecting curves. Using
that notion of angle\ we see that (\ref{kos2}) is also invariant.
\end{proof}

In some cases there is a simple relation between the values of
$\operatorname*{kos} $ and the geometry of triangles.

Given a triangle $T=\left\{  x_{1},x_{2},x_{3}\right\}  \subset$
$\mathbb{CH}^{2}$ we use an automorphism to suppose that $T$ is in
$\mathbb{CH}^{2}$ and $x_{1}$ is at the origin. Regard $\mathbb{CH}^{2}$ as
$\mathbb{B}_{2}$ inside $\mathbb{C}^{2},$ denote by $V$ the real linear span
of the points of $T,$ and set $W=V\cap\mathbb{B}_{2}.$ It may be that $W$ is a
totally geodesic submanifold of $\mathbb{CH}^{2}.$ In that case, with the
origin in $W,$ there are three possibilities. First, $W$ may be a classical
geodesic in which case it will be a line segment through the origin. Second,
$W$ may be a complex geodesic which contains the origin. We can suppose it is
the unit disk in the plane of the first coordinate and hence it is the
classical Poincar\'{e} disk and it has constant negative curvature $-1$. The
final possibility is that $W$ is a totally real totally geodesic disk and
hence, after an automorphism, we can suppose it is $BK_{2}=\left\{  \left(
r,s\right)  \in\mathbb{B}_{2}:r,s\in\mathbb{R}\right\}  ,$ the Beltrami-Klein
disk of constant curvature $-1/4.$

In the first case, noting (\ref{kos1}) and (\ref{kos2}) we then see that
$\operatorname*{kos}\nolimits_{1}(2,3)=\pm1.$ The value $-1$ occurs when
$x_{1}$ separates the other two points, the value $+1$ when it does not. In
the second case $x_{3}$ is in $G(1,2)$ and so $P_{G(1,2)}x_{3}=x_{3}.$ From
(\ref{kos1}) we see that $\left\vert \operatorname*{kos}\nolimits_{x_{1}%
}(x_{2},x_{3})\right\vert =1$ and then from (\ref{kos2}) that
$\operatorname*{kos}\nolimits_{1}(2,3)=e^{i\gamma}$ where $\gamma$ is the
Euclidean angle between the segments $x_{1}x_{2}$ and $x_{1}x_{3}.$ Thus
$\gamma$ is the Euclidean and the hyperbolic angle at vertex $x_{1}.$
Similarly the values of $\operatorname*{kos}$ at the other two vertices give
the other angles. The congruence class of a triangle in a plane of constant
negative curvature is determined by its angles and hence in this case also by
the three values of $\operatorname*{kos}$. Finally, in the third case the
triangle is in a totally real vector space and we are in the situation
discussed in Proposition \ref{real kos} below. From (\ref{simple kos}) we see
that $\operatorname*{kos}\nolimits_{x_{1}}(x_{2},x_{3})=\cos\theta$ where
$\theta$ is the hyperbolic angle of the triangle at $x_{1}.$ Similarly for the
other two vertices. Hence, as in the previous case, we know the three angles
of a triangle in a plane of constant negative curvature and that determines
the congruence class of the triangle. Finally, looking backward we see that
the first case is the second and third cases holding simultaneously.

In each of these cases the argument can be reversed. If $\operatorname*{kos}%
\nolimits_{1}(2,3)=\pm1$ then $W$ is a line through the origin, if $\left\vert
\operatorname*{kos}\nolimits_{x_{1}}(x_{2},x_{3})\right\vert =1$ then
$P_{G(1,2)}x_{3}=x_{3}$ and hence the triangle lies in a complex geodesic, and
finally, noting Proposition \ref{real} below, if $\operatorname*{kos}%
\nolimits_{x_{1}}(x_{2},x_{3})$ is real then after automorphism the triangle
is in the position described.

\subsection{$Kos$ and other Invariants\label{invariants}}

It is a result of Brehm \cite{Br} that equality of the data sets $S^{\prime} $
is a congruence criterion for triangles in $\mathbb{CH}^{n}$. That criterion
and variations are often used in describing the geometry of finite sets in
$\mathbb{CH}^{n},$ \cite{BE}, \cite{HS}, \cite{G}, \cite{CG}, \cite{Ro}; but
here we focus on the congruence criterion given by the data set $S^{\prime
\prime}.$ It is straightforward to pass between $S^{\prime}$ and
$S^{\prime\prime}.$ The parameters are invariant so we can assume we are in
the model case and the triangle is $\Gamma=\left\{  x_{1},x_{2},x_{3}\right\}
=\left\{  \left(  0,0\right)  ,\left(  a,0\right)  ,\left(  x,b\right)
\right\}  .$

First suppose we have the data $S^{\prime\prime},$ that is $a=\delta
(x_{1},x_{2}),$ $\omega=\delta(x_{1},x_{3}),$ and $\operatorname*{kos}%
\nolimits_{1}(2,3)$ are known. Using those values and (\ref{simple kos}) we
also know $a\bar{x}$. Hence using (\ref{2.1111}) we can find the third side
length using \
\begin{equation}
\delta^{2}(x_{2},x_{3})=1-\frac{(1-\left\Vert x_{2}\right\Vert ^{2}%
)(1-\left\Vert x_{3}\right\Vert ^{2})}{\left\vert 1-\left\langle \left\langle
x_{2},x_{3}\right\rangle \right\rangle \right\vert ^{2}}=1-\frac{(1-\alpha
^{2})(1-\omega^{2})}{\left\vert 1-\alpha\bar{x}\right\vert ^{2}}.
\label{length}%
\end{equation}
To finish determining $S^{\prime}$ we need to find $\beta,$ the angular
invariant. For this triangle $\beta=\arg(1-a\bar{x})$ and, as we just
mentioned, $a\bar{x}$ is known; hence $\beta$ is known.

In the other direction, to go from $S^{\prime}$ to $S^{\prime\prime}$ we need
to find $\operatorname*{kos}_{1}(2,3).$ In this case we know the three side
lengths and hence, noting (\ref{length}), we know $\left\vert 1-\alpha\bar
{x}\right\vert .$ We also know the angular invariant $\beta=\arg(1-a\bar{x}).
$ Combining these two we know $a\bar{x}.$ With that information, and using
(\ref{computed kos}), we can find $\operatorname*{kos}_{1}(2,3).$

Given points $\left\{  x_{1},x_{2},x_{3}\right\}  $ we are the describing the
geometry of the intersection of the geodesics $x_{1}x_{2}$ and $x_{1}x_{3}$ at
$x_{1}$ using the complex number $\operatorname*{kos}_{1}(2,3)$. There are
other complex parameters and pairs of real parameters that are used for the
same purpose. For instance in his geometric analyses of the ball model of
$\mathbb{CH}^{n}$ Goldman uses angles, $\phi,\theta$ related to
$\operatorname*{kos}$ by $\cos\phi=\left\vert \operatorname*{kos}%
_{1}(2,3)\right\vert $ and $\cos\theta=\operatorname{Re}\operatorname*{kos}%
_{1}(2,3)$ \cite[pg. 88]{Go}.

\subsection{$Kos$ and Vertices}

For a bivalent vertex in $\mathbb{R}^{n}$ the cosine of the vertex angle
carries all the intrinsic geometric data about the vertex. The geometry of a
trivalent vertex is determined by the geometry of the three component bivalent
vertices. The values of $\operatorname*{kos}$ provide similar information for
vertices in $\mathbb{CH}^{n}.$

First we consider sets in $\mathbb{R}^{n}$. By a \textit{vertex\textit{\ }%
}$\mathbf{V}$\textit{\ }in $\mathbb{R}^{n}$ we mean a collection of two or
more line segments, \textit{rays}, with a common starting point the
\textit{vertex point}, $V,$ also simply called the vertex. We call
$\mathbf{V}$ \textit{bivalent} if it has two rays, \textit{trivalent} if there
are three, \textit{multivalent} in general. We suppose the rays of a bivalent
vertex are ordered. The two rays of a bivalent vertex span an affine plane and
form an angle in that plane, the \textit{vertex angle}. We say two vertices
are (Euclidean) congruent if there is a Euclidean isometry placing the second
vertex point at the same position as the first and so that the initial
segments of the rays of the second coincide with the initial segments of the
rays of the first.

\begin{proposition}
Two bivalent vertices in $\mathbb{R}^{n}$ are congruent if and only if the
cosines of their vertex angles are equal.
\end{proposition}

Suppose now we are given $\mathfrak{S}$, a set of three triangles $\left\{
T_{i}\right\}  _{i=1}^{3}$ each contained in some $\mathbb{R}^{n}.$ Suppose
for $i=1,2,3$ that $T_{i}$ has vertices $\left\{  x_{i1},x_{ia},x_{ib}%
\right\}  ,$ that the Euclidean length of the side $x_{i1}x_{ia}$ is $l_{ia}$
and similarly for $l_{ib};$ and that $\mathbf{W}_{i}$ is the bivalent vertex
in $T_{i}$ with vertex point $x_{i1}$ and rays ordered similarly to the sides
of the triangle; $x_{i1}x_{ia}$ first, $x_{i1}x_{ib}$ second. It may or may
not be possible to join these three triangles as faces of a tetrahedron in
some $\mathbb{R}^{n}.$ That is there may be a tetrahedron $\left\{
y_{1},y_{2},y_{3},y_{4}\right\}  $ in some $\mathbb{R}^{n}$ with the
triangular face $\left\{  y_{1},y_{2},y_{3}\right\}  $ congruent to $T_{1},$
$\left\{  y_{1},y_{3},y_{4}\right\}  $ congruent to $T_{2},$ and $\left\{
y_{1},y_{2},y_{4}\right\}  $ congruent to the triangle $\widetilde{T_{3}}=$
$\left\{  x_{31},x_{3b},x_{3a}\right\}  .$ (That last triangle has the same
vertices as $T_{3}$ but their order is different. That distinction will be
significant when we consider the complex case.) Certainly a necessary
condition for building a tetrahedron is that certain side lengths match.

\begin{definition}
\label{set}The\textit{\ triangles} $\mathfrak{S}$ are said to be a matched set
if there are numbers $\left\{  L_{i}\right\}  _{i=1}^{3}$ such that
\begin{equation}
l_{1b}=l_{2a},=L_{1},\text{ }l_{2b}=l_{3a}=L_{2},\text{ }l_{3b}=l_{1a}=L_{3}.
\label{match}%
\end{equation}

\end{definition}

\begin{proposition}
\label{real tetrahedron}Given a matched set of three Euclidean triangles
$\mathfrak{S}$ the following are equivalent.

\begin{enumerate}
\item The triangles of $\mathfrak{S}$ are congruent to the faces of a tetrahedron.

\item If $\mathfrak{S}^{\prime}$ is another matched set of three triangles
(with associated data denoted by primes) and for $i=1,2,3$ the bivalent vertex
$\mathbf{W}_{i}^{\prime}$ is congruent to the bivalent vertex $\mathbf{W}_{i}$
then the triangles of $\mathfrak{S}^{\prime}$ are congruent to the faces of a
tetrahedron. That is, the previous conclusion holds for any choice of the
lengths $\{L_{i}^{^{\prime}}\}.$

\item There is a trivalent vertex $\mathbf{V}$ whose three component bivalent
vertices are congruent to the three bivalent vertices $\left\{  \mathbf{W}%
_{i}\right\}  _{i=1}^{3}.$

In each case the tetrahedron or the trivalent vertex is unique up to congruence.
\end{enumerate}
\end{proposition}

In sum, if the side lengths match then the possibility of putting the
triangles together is determined by the geometric combinatorics of the vertex
angles, the specific lengths $\left\{  L_{i}\right\}  $ play no role.

Analogous results hold in $\mathbb{CH}^{n}.$ By a \textit{vertex }$\mathbf{V}
$\textit{\ }in $\mathbb{CH}^{n}$ we mean a collection of two or more geodesic
segments, \textit{rays}, $R_{i},$ with a common starting point the
\textit{vertex point}, $V.$ If there are two rays we assume they are ordered.
Again a vertex may be \textit{bivalent, trivalent, }or \textit{multivalent.
}We say two vertices are congruent if there is an automorphism placing the
second vertex point on the first and so that the initial segments of the two
sets of rays overlap.

The geometry of a bivalent vertex in $\mathbb{CH}^{n}$ can be described by two
real numbers or one complex number. We will use the complex quantity
$\operatorname*{kos}\left(  \mathbf{V}\right)  $ which we define to be
$\operatorname*{kos}_{V}(x_{1},x_{2})$ where $V$ is the vertex point and
$x_{i}$ is chosen on the ray $R_{i},$ $i=1,2.$ The formula (\ref{simple kos})
shows that this value does not depend on those choices. Thus if $\left\{
x_{1},x_{2},x_{3}\right\}  $ is a triangle and $\mathbf{V}$ the bivalent
vertex with vertex point $x_{1}$ then $\operatorname*{kos}\left(
\mathbf{V}\right)  =\operatorname*{kos}_{1}(2,3).$

If we regard $\operatorname*{kos}\left(  \mathbf{V}\right)  $ as a substitute
for the cosine of the angle of a bivalent vertex in Euclidean space then
statement (8) in Theorem \ref{triangle} is the analog of the classical
side--angle--side congruence criterion for Euclidean triangles. We also have
analogs of the two previous results.

\begin{proposition}
\label{marker}Two bivalent vertices,$\mathbf{V}$, $\mathbf{W,}$ in
$\mathbb{CH}^{n}$ are congruent if and only if $\operatorname*{kos}\left(
\mathbf{V}\right)  =\operatorname*{kos}\left(  \mathbf{W}\right)  .$
\end{proposition}

\begin{proof}
If the vertices are congruent then the conclusion is clear. In the other
direction, pick $\varepsilon$ small and pick $x$ and $y$ on the two rays of
$\mathbf{V}$ at distance $\varepsilon$ from the vertex point. Pick $u$ and $v$
similarly on the rays of $\mathbf{W.}$ From (8) in Theorem \ref{triangle} we
see that the triangle formed by the vertex point of $\mathbf{V}$ together with
$x$ and $y$ is congruent to the triangle formed by the vertex point of
$\mathbf{W}$ together with the points $u,v.$ That congruence of triangles also
gives the required congruence of the vertices
\end{proof}

There is an interesting contrast between the real and complex cases. For a
bivalent vertex $\mathbf{V}$ let $\mathbf{V}_{\mathtt{rev}}$ be the vertex
obtained from $\mathbf{V}$ by reversing the order of the two rays. For
vertices in $\mathbb{RH}^{n}$ $\mathbf{V}_{\mathtt{rev}}$ is congruent to
\textbf{V. }However if $\mathbf{V\subset}$ $\mathbb{CH}^{n}$ then
$\mathbf{V}_{\mathtt{rev}}$ is congruent to $\mathbf{V}^{\ast},$ the vertex
obtained by conjugating the coordinates of $\mathbf{V,}$ but generally is not
congruent to $\mathbf{V.}$ To see this, compute that $\operatorname*{kos}%
(\mathbf{V}_{\mathtt{rev}}\mathbf{)}=\overline{\operatorname*{kos}\left(
\mathbf{V}\right)  }= $ $\operatorname*{kos}(\mathbf{V}^{\ast}\mathbf{)}$.

Suppose now we are given $\mathfrak{S,}$ a set of three triangles each in
$\mathbb{CH}^{n}$. We continue the earlier notation and terminology, but now
we use the pseudohyperbolic side lengths $\delta(\cdot,\cdot)$. With that
change the analog of Proposition \ref{real tetrahedron} holds.

\begin{proposition}
\label{vertex}Given a matched set of triangles $\mathfrak{S}$ in
$\mathbb{CH}^{n}$ the statements (1), (2), and (3) of Proposition
\ref{real tetrahedron} are equivalent. If the conditions hold then the
congruence class of the tetrahedra and of the trivalent vertex are uniquely determined.
\end{proposition}

\begin{proof}
It is immediate that $(2)\Longrightarrow(1)\Longrightarrow(3).$ To see that
$(1)\Longrightarrow(2)$ suppose from $(1)$ that we have $\mathfrak{S}$ and the
associated tetrahedron $\Lambda$ and suppose that we are given a new set of
lengths $\{L_{i}^{^{\prime}}\}$ from $(2).$ Use an automorphism to replace
$\Lambda$ with the tetrahedron $\Delta$ with coordinates given by
(\ref{tetra}). For $i=2,3,4$ select $\gamma_{i}>0$ so that the tetrahedron
$\Delta^{\prime}=\{x_{1},\gamma_{2}x_{2},\gamma_{3}x_{3},\gamma_{4}x_{4}\}$
has $\left\Vert \gamma_{i}x_{i}\right\Vert =L_{i}^{\prime}$. The bivalent
vertices of $\Delta^{\prime}$ at the origin are obtained from those of
$\Delta$ by changing the lengths of the rays which does not change the
congruence class of the vertices. Hence those vertices have the desired
congruence classes. Also the lengths of the edges of $\Delta^{\prime}$ which
meet at the origin match the $\left\{  L_{i}^{\prime}\right\}  .$ Thus by the
results on $S^{\prime\prime}$ in Theorem \ref{triangle} the triangles of
$\mathfrak{S}^{\prime}$ are congruent to the faces of $\Delta^{\prime},$
establishing $(2).$ To show that $(3)\Longrightarrow(1)$ first use an
automorphism to place the trivalent vertex $\mathbf{V}$ at the origin with its
rays in the directions of the rays of $\Delta$ of (\ref{tetra}). Next select a
point on each ray whose distance from the origin is the appropriate $L_{i}.$
The triangles with vertex at the origin are, again by Theorem \ref{triangle},
congruent to the triangles of $\mathfrak{S}$ and hence the origin together
with those three new points are the vertices of the tetrahedron required to
show that (1) holds.

For uniqueness, first consider two trivalent vertices $\mathbf{W}$ and
$\mathbf{W}^{\prime}$ which satisfy Condition (3). Pick a small $\varepsilon$
and a pick point on each ray at distance $\varepsilon$ from the vertex. Let
$\Sigma$ be the tetrahedron determined by those four points and let
$\Sigma^{\prime}$ be the similar tetrahedron constructed using $\mathbf{W}%
^{\prime}.$ We will be done if we show $\Sigma$ and $\Sigma^{\prime}$ are
congruent. The argument which shows they are congruent also gives the
uniqueness in statements (1) and (2). First note that the results on
$S^{\prime\prime}$ in Theorem \ref{triangle} insures that the triangular faces
of $\Sigma$ meeting at the vertex point $\mathbf{W}$ are congruent to those in
$\Sigma^{\prime}$ meeting at the vertex point of $\mathbf{W}^{\prime}.$ By
Condition (5) of Theorem \ref{reduction} this is enough to show the tetrahedra
are congruent.
\end{proof}

Thus, as in $\mathbb{R}^{n}$, the possibility of putting the triangles
together is determined by the geometric combinatorics of the vertex angles,
now described by $\operatorname*{kos}.$ Again the lengths $\left\{
L_{i}\right\}  $ play no role.

In the Corollary \ref{vertexx} we will see a similar result for congruence of
multivalent vertices in $\mathbb{CH}^{n}.$

\subsection{$Kos$ and Congruence of Triangles}

Angles are commonly used in discussing congruence of Euclidean triangles but
the angle cosines are an equivalent parameter \ Motivated by the analogy
between $\cos$ and $\operatorname*{kos}$ we briefly look at the role of
$\operatorname*{kos}$ in congruence criteria for triangles in complex
hyperbolic space.

Recall the standard naming conventions for Euclidean triangles. The
abbreviation SAS refers to the length of two sides and the size of the angle
between them, similarly for the other combinations of S and A. The Euclidean
results are that SSS, SAS, ASA, and AAS each give enough information to
determine the congruence class of a triangle, the data SSA can give two
different congruence classes, and AAA only determines the triangle up to
similarity. These and related topics, for Euclidean, real hyperbolic, and
spherical geometries are discussed in detail in Section 7 of \cite{J}.

From \cite{J} we see that the results for triangles in $\mathbb{RH}^{2}$ are
the same as for $\mathbb{R}^{2}$ with one exception, in $\mathbb{RH}^{2}$ the
data set AAA determines the congruence class of a triangle. The version of
$\mathbb{RH}^{2}$ with curvature $-1$ is isometric to $\mathbb{CH}^{1}$ and
knowing the value of $\operatorname*{kos}$ at a vertex of a triangle in
$\mathbb{CH}^{1}$ is equivalent to knowing the angle of the vertex of that
triangle in $\mathbb{RH}^{2}$. Hence all the results mentioned for
$\mathbb{RH}^{2}$ also hold for triangles in $\mathbb{CH}^{1}.${}

The results for triangles in $\mathbb{RH}^{2}$ also hold for triangles in
$\mathbb{RH}^{n}.$ To see this note that any triangle in $\mathbb{RH}^{n}$ is
inside a totally geodesically embedded copy of $\mathbb{RH}^{2}$ and the
analysis can be done in that copy. However the analogous statement in complex
hyperbolic space is not true; not every triangle in $\mathbb{CH}^{n}$ is
contained in a complex geodesic and hence there is no natural notion of the
cosine of the angle at a vertex. We will try to proceed using the value of
$\operatorname*{kos}$ at a vertex as a substitute. A parameter count shows
this situation will be different. We see from Theorem \ref{2} that the
congruence class of a triangle in a complex geodesic is determined by 3 real
parameters, but to specify the class of a general triangle requires 4
parameters. Hence it is not surprising that SSS, which only has three
parameters, does not give a congruence criterion for triangles in
$\mathbb{CH}^{n}.$ To see the failure for triangles in $\mathbb{CH}^{2}$
consider the triangles
\[
T_{t}=\left\{  \left(  0,0\right)  ,\left(  \frac{1}{2},0\right)  ,\left(
\frac{te^{i\theta}}{2},\frac{(2-t^{2})^{1/2}}{2}\right)  \right\}  ,\text{
}1\leq t\leq\sqrt{2},\text{ }\cos\theta=\frac{t^{2}+7}{8t}.
\]
Computation shows they all have the same side lengths, however triangles in
the form (\ref{tri}), as these are, are only congruent if they are identical.

On the other hand the SAS data set for triangles in $\mathbb{CH}^{n}$ does
determine their equivalence class. This is a consequence of Theorem \ref{2}
because the data set SAS is the data set $S^{\prime\prime}$ of that theorem.

In summary, the congruence class of a triangle in $\mathbb{CH}^{1}$ is
determined by each of the data sets mentioned except SSA. In $\mathbb{CH}%
^{n},n>1$ SAS is sufficient to determine the congruence class, neither SSA nor
SSS is sufficient, and the situation with the other data sets, ASA, AAS, and
AAA, is not known.

\subsection{$Kos$ and Area\label{area}}

Let $T=\left\{  x,y,z\right\}  $ be the vertices of a triangle in
$\mathbb{CH}^{1}$ with geodesic segments as sides. The invariant area of $T,$
$\operatorname*{Area}(T),$ can be evaluated using the angular invariant,
2$\alpha(x,y,z)=$ $\operatorname*{Area}(T)$ \cite[1.3.6]{Go}, and by
invariance a similar result holds if $T$ is contained in a complex geodesic.
The next proposition and corollary hold for triangles in a complex geodesic
but for convenience we present it for sets in $\mathbb{CH}^{1}.$

\begin{proposition}
If $T=\left\{  x_{1},x_{2},x_{3}\right\}  \subset\mathbb{CH}^{1}$ then
\begin{equation}
\pi-\arg\left(  \operatorname*{kos}\nolimits_{1}(2,3)\operatorname*{kos}%
\nolimits_{2}(3,1)\operatorname*{kos}\nolimits_{3}(1,2)\right)  =2\alpha
(x_{1},x_{2},x_{3})=\operatorname*{Area}(T) \label{goal2}%
\end{equation}

\end{proposition}

\begin{proof}
We identify $\mathbb{CH}^{1}$ with the unit disk in the complex plane. From
(\ref{first kos}) we see that $\operatorname*{kos}_{1}(2,3)$ is a positive
multiple of%
\[
\kappa_{123}=\left\langle \left\langle \phi_{x_{1}}(x_{2}),\phi_{x_{1}}%
(x_{3})\right\rangle \right\rangle .
\]
Hence in computing the left hand side, $LHS$, of (\ref{goal2})\ we can replace
$\operatorname*{kos}(2,3)$ with $\kappa_{123}$ and similarly for other
indices. The conformal involutions of the disk are given by Blaschke factors.
Hence, noting that for $a,b$ in the disk $\left\langle \left\langle
a,b\right\rangle \right\rangle =a\bar{b}$ we find that
\[
\kappa_{123}=\frac{x_{1}-x_{2}}{1-x_{1}\overline{x_{2}}}\overline{\frac
{x_{1}-x_{3}}{1-x_{1}\overline{x_{3}}}}%
\]
Thus%
\[
LHS=\pi-\arg\frac{-\Pi\left\vert x_{i}-x_{j}\right\vert ^{2}}{\Pi\left(
1-x_{i}\overline{x_{j}}\right)  ^{2}}%
\]
with both products over the index pairs $\left\{  \left(  1,2\right)  ,\left(
2,3\right)  ,\left(  3,1\right)  \right\}  .$ The positive factor
$\Pi\left\vert x_{i}-x_{j}\right\vert ^{2}$ does not affect the value of
$\arg$ and hence we continue with
\[
LHS=\pi-\left(  \pi+2\arg\Pi k(x_{i},x_{j}\right)  )=2\alpha
\]
the last equality by (\ref{def ang}).

To finish we need to show that one of the first two terms in (\ref{goal2})
equals $\operatorname*{Area}(T).$ In fact there are separate reasons why each
equals $\operatorname*{Area}(T).$ We mentioned that the relation between
$\alpha$ and $\operatorname*{Area}(T)$ is in \cite{Go}. For the other case let
$\gamma_{i}$ be the angle at vertex $x_{i}$. We see from Proposition
\ref{formula} that because $T\subset\mathbb{CH}^{1}$ we have
$\operatorname*{kos}_{1}(2,3)=e^{i\gamma_{1}}$ where $\gamma_{1}$ is the angle
at $x_{1}$ of the triangle $T. $ Similarly for the other indices. Using that
we see that $LHS$ in (\ref{goal2}) equals $\pi-(\gamma_{1}+\gamma_{2}%
+\gamma_{3})$ which equals $\operatorname*{Area}(T)$ by the classical result
based on the Gauss-Bonnet theorem$.$
\end{proof}

This result extends in the usual way to convex hyperbolic polygons. The
interior of the polygon can be triangulated, the area of each triangular piece
computed using the expressions in (\ref{goal2}) and the total area computed as
a sum. There will be cancellations generated by the fact that for kernel
functions $k_{xy}=\overline{k_{yx}}$. The final result will involve the
boundary vertices considered cyclically. Let $\operatorname*{Area}%
(x_{1},...,x_{n})$ denote the invariant area of the interior of the polygon
with vertices $\left\{  x_{i}\right\}  _{i=1}^{n}\subset\mathbb{CH}^{1}$ and
with the indices continued cyclically, $x_{n+1}=x_{1},$ etc. We have

\begin{corollary}%
\[
(n-2)\pi-\arg\prod\limits_{i=1}^{n}\operatorname*{kos}\nolimits_{i}%
(i+1,i+2)=-2\arg\prod\limits_{i=1}^{n}k_{x_{i}x_{i+1}}=\operatorname*{Area}%
(x_{1},...,x_{n}).
\]

\end{corollary}

If $T$ is not in a complex geodesic then the first two terms in (\ref{goal2})
need not be equal. In that more general case 2$\alpha$ equals the symplectic
area of $T$, the integral of the symplectic form of $\mathbb{CH}^{n}$ over a
real two manifold bounded by the sides of $T$ \cite{HM}. That quantity is also
equal to the area of the triangle in the complex geodesic $G $ containing
$x_{1}$ and $x_{2}$ and with vertices $x_{1},x_{2},$ and $P_{G}x_{3}$
\cite{Go}. There is a nice discussion of the geometry associated to the
angular invariant in \cite{CO} and that paper presents the result we used for
triangles in $\mathbb{CH}^{1}$ as the simplest special case of the rich
relationship between the angular invariant and Hermitian geometry. More about
that relationship is in \cite{BI}, \cite{BIW}, \cite{C}, \cite{CO}, \cite{HM},
\cite{Go}, \cite{Ro}, and the references there.

In the previous proposition (\ref{goal2}) can be rewritten using
$\arg(abc)=\operatorname{Im}\log a+\operatorname{Im}\log b+\operatorname{Im}%
\log c.$ If that is done then the following proposition appears to somehow be
a pair with the previous one. As before the result can be formulated
invariantly but we just present the cleaner case. The set $BK_{2}=\left\{
\left(  x,y,0,...,0\right)  \in\mathbb{B}_{n}:x,y\in\mathbb{R}\right\}  $
inside $\mathbb{CH}^{n}=\mathbb{B}_{n}$ is a totally geodesically embedded
copy of the Beltrami-Klein model of $\mathbb{RH}^{2}$ which has constant
curvature $-1/4.$ As such it carries a natural area measure.

\begin{proposition}
Suppose $T=\left\{  x_{1},x_{2},x_{3}\right\}  \subset BK_{2}.$ Then%
\[
4(\pi-(\cos^{-1}\operatorname*{kos}\nolimits_{1}(2,3)+\cos^{-1}%
\operatorname*{kos}\nolimits_{3}(1,2)+\cos^{-1}\operatorname*{kos}%
\nolimits_{2}(3,1)))=\operatorname*{Area}(T).
\]

\end{proposition}

\begin{proof}
This is a consequence of two facts. First, taking note of (\ref{simple kos}),
$\cos^{-1}\operatorname*{kos}(2,3)$ is the angle between the geodesics
$x_{1}x_{2}$ and $x_{1}x_{3},$ and similarly for the other indices. Secondly,
by the Gauss-Bonnet theorem, the area of a triangle with angles $\alpha
,\beta,\gamma$ in a plane of constant curvature $-1/4$ is $4(\pi-(\alpha
+\beta+\gamma)).$
\end{proof}

It would be interesting to have a general result which unifies the two propositions.

\subsection{$Kos$ and the Multiplier Algebra}

We discuss briefly a role of $\operatorname*{kos}$ the study of multiplier
algebras of RKHSI; more information is in \cite{Ro} and \cite[Sec. 3]{Ha}.

Suppose $H\in\mathcal{RK}$ and let $\mathcal{M}=\mathcal{M}(H)$ be the algebra
of multiplier operators on $H$ normed by the operator norm. Introduce
coordinates on $\mathcal{M}$ by associating the multiplier $M$ with the $n$ -
tuple $\lambda=\left(  \lambda_{1},...,\lambda_{n}\right)  $ where $M^{\ast
}k_{j}=\overline{\lambda_{j}}k_{j},$ $j=1,...,n.$ Let $\left(  \mathcal{M}%
\right)  _{1}$ be the unit ball of $\mathcal{M}$ and let
$\operatorname*{Slice}_{1}\left(  \left(  \mathcal{M}\right)  _{1}\right)  $
be the slice $\left(  \mathcal{M}\right)  _{1}\cap\left\{  \lambda:\lambda
_{1}=0\right\}  .$ The discussion in \cite{Ro} and in \cite[Sec. 3]{Ha} gives
the following.

\begin{proposition}
[{\cite[Sec. 3]{Ha}}]If $H\in\mathcal{CPP}$ then the point on
$\operatorname*{Slice}_{1}\left(  \left(  \mathcal{M}\right)  _{1}\right)  $
with maximum value of $\operatorname{Re}\lambda_{j} $ , $2\leq j\leq n,$ has
coordinates
\[
\left(  0,\delta_{12}\operatorname*{kos}\nolimits_{1}(j,2),,...,\delta
_{1n}\operatorname*{kos}\nolimits_{1}(j,n)\right)  .
\]

\end{proposition}

Thus if $H=DA(X)\ $then the geometry of $\left(  \mathcal{M}(H)\right)  _{1}$
is enough to reconstruct the Gram matrix $\operatorname*{Gr}(H)$ and hence
also enough to describe $H$ and $X.$

\section{Finite Sets in $\mathbb{CH}^{n}$\label{point sets}}

\subsection{Describing Sets by Their Triangles}

It is possible to give parametric descriptions of congruence classes of sets
$X=\left\{  x_{i}\right\}  _{i=1}^{n+1}\subset\mathbb{CH}^{n}$ using the
distances $\delta(i,j)$ and the angular invariants $\alpha(i,j,k),$ \cite{BE},
\cite{HS}, \cite{G}, \cite{CG} \cite{Ro}. The description of triangles using
the data set $S^{\prime}$ is the simplest example. Using those results and
their proofs it is also possible to specify a unique model set in each
congruence class, the triangles $\Gamma$ in (\ref{tri}) and the tetrahedra
$\Delta$ in (\ref{tetra}) are examples. Those invariants and related ones can
also be us to describe more general geometric structures associated to
$\mathbb{CH}^{n},$ see for instance \cite{C} and \cite{CG}.

When those descriptions are specialized to sets in real hyperbolic space,
$X\subset\mathbb{RH}^{n}\subset\mathbb{CH}^{n},$ the angular invariant
vanishes and we obtain descriptions of polyhedra in $\mathbb{RH}^{n}$ in terms
of edge lengths. Here we use $\operatorname*{kos}$ rather than the angular
invariant and obtain descriptions of sets in $\mathbb{RH}^{n}$ and
$\mathbb{CH}^{n}$ which emphasize a different type of geometric data; the
basic example is the use of $S^{\prime\prime}$ to describe triangles. Our
description generalizes the idea of describing polyhedra in $\mathbb{RH}^{n}$
using a mix of edge lengths and cosines of vertex angles. The approach is well
suited to describing the relation between the geometry of a set and its
subsets$.$ In particular it lets give explicit answers to the questions in the introduction.

Suppose we are given $X=\left\{  x_{i}\right\}  _{i=1}^{n+1}$ $\subset
\mathbb{CH}^{n}$ and we connect $x_{1}$ to each of the other $x_{i}$ by a
geodesic $\gamma_{1i}.$ The point $x_{1}$ will be the vertex point of an
$n$--valent vertex $\mathbf{V}$\textbf{\ }which is composed of the bivalent
vertices $\mathbf{V}_{ij},$ $2\leq i,j\leq n+1$ having $\gamma_{1i}$ as a
first ray and\ $\gamma_{1j}$ as a second. We will describe $X$ using the
distances between $x_{1}$ and the other points and the numbers $K_{ij}%
=\operatorname*{kos}(\mathbf{V}_{ij})$. Specifically, recalling the definition
(\ref{KOS}), we define the $n$-vector $\varrho(X)$ and the $n\times n$ matrix
$\mathcal{M}(X)$ by%
\begin{align}
\rho(X)  &  =\left(  \delta(x_{1},x_{2}),...,\delta(x_{1},x_{n+1})\right)
,\text{ and}\label{m=}\\
\mathcal{M}(X)  &  =\left(  K_{ij}\right)  _{i,j=2}^{n+1}=\left(
\operatorname*{kos}\left(  \mathbf{V}_{ij}\right)  \right)  _{i,j=2}%
^{n+1}=\operatorname*{KOS}(DA(X),1).\nonumber
\end{align}

\begin{theorem}
\label{general} Given $X=\left\{  x_{i}\right\}  _{i=1}^{n+1}\subset
\mathbb{CH}^{n}$:

\begin{enumerate}
\item Each entry of $\varrho(X)$ is between $0$ and 1, $\mathcal{M}(X)$ has
$1$'s on the diagonal and $\mathcal{M}(X)\succcurlyeq0$.

\item Conversely, given an $n$-tuple $\sigma$ of numbers between 0 and $1$,
and a matrix $\mathcal{N}$ with $1^{\prime}s$ on the diagonal and
$\mathcal{N}$ $\succcurlyeq0$, there is an $X\subset\mathbb{CH}^{n}$ with
$\mathcal{M}(X)$ $=\mathcal{N}$ and $\varrho(X)=\sigma.$

\item Given $Y\subset\mathbb{CH}^{n}$, $X\sim Y$ if and only if $\mathcal{M}%
(X)=\mathcal{M}(Y)$\ and $\varrho(X)=\varrho(Y).$
\end{enumerate}
\end{theorem}

\begin{proof}
In (1) the claim for $\rho(X)$ is clear. The matrix $\mathcal{M}(X)$ is
invariant under automorphisms of $\mathbb{CH}^{n}$ and hence we can suppose
$x_{1}$ is at the origin. Let $\widehat{X}$ be the\ set of radial projections
of the remaining points onto the unit sphere, $\widehat{X}$ $=\{\widehat{x_{2}%
},...,\widehat{x_{n+1}}\}\subset\mathbb{\partial B}_{n}.$ We then see from
(\ref{simple kos}) that $K_{ij}=\left\langle \left\langle \widehat{x_{i}%
},\widehat{x_{j}}\right\rangle \right\rangle $ for $2\leq i,j\leq n+1.$ Thus
$\mathcal{M}(X)$ is the Gram matrix of the set of vectors $\widehat{X}$ and
hence is positive semidefinite. For (2), a matrix with the properties of
$\mathcal{N}$ must be the Gram matrix of a set $W=\left\{  w_{i}\right\}
_{i=2}^{n+1}\subset\mathbb{C}^{n}$ , unique up to unitary equivalence. The
$1$'s on the diagonal of $\mathcal{N}$ insure that $W\subset\partial
\mathbb{B}_{n},$ We now form $X$ by designating the origin as $x_{1}$ and for
$2\leq i\leq n+1$ picking $x_{i}$ on the line segment $\left[  0,w_{i}\right]
$ with $\left\vert x_{i}\right\vert =\sigma_{i}.$ It is straightforward that
$X$ has the required properties.

For (3), first suppose $Y$ is congruent to $X;$ that is $Y$ is the image of
$X$ under an automorphism of the ball. The entries of $\varrho$ and
$\mathcal{M}$ are automorphism invariants and this gives the desired
equalities. In the other direction suppose the data associated with $X$ equals
the data associated with $Y.$ Without loss of generality we can suppose
$x_{1}$ is at the origin in which case $\mathcal{M}(X)$ is the Gram matrix of
the set $\hat{X}\subset\partial\mathbb{B}_{n}$. Similarly for $Y$ and $\hat
{Y}.$ Thus $\hat{X}$ and $\hat{Y}$ have the same Gram matrix and hence there
is a unitary map $\mathbb{C}^{n}$ which takes $\hat{X}$ to $\hat{Y}.$ That
unitary is also an automorphism $\mathbb{CH}^{n}$ and so takes each geodesic
from the origin to a point of $\hat{X}$ to a\ geodesic connecting the origin
to a point of $\hat{Y}.$ Given the further assumption that $\varrho
(X)=\varrho(Y)$ it must take $X$ to $Y,$ as required.
\end{proof}

It is possible to start with this result and describe a unique $X$ in each
congruence class, for instance by moving $X$ so that $x_{1}$ is at the origin
and then constructing an orthonormal basis for $\mathbb{C}^{n}$ in which the
matrix of coordinates of $\hat{X}$ is lower triangular. Theorem 7 of \cite{Ro}
is an instance of this\ approach carried out in detail.

\begin{corollary}
\label{our mq}If $H\in\mathcal{RK}$ then $H\in\mathcal{CCP}$ is and only if
$\operatorname*{KOS}(H,1)\succcurlyeq0.$
\end{corollary}

\begin{proof}
This is immediate from the previous theorem and Theorem \ref{reduction}.
\end{proof}

We say $X$ is in \textit{general position} if it is not contained in a totally
geodesic copy of $\mathbb{CH}^{n-1}$ inside of $\mathbb{CH}^{n}.$ From the
previous theorem and proof we have

\begin{corollary}
$X$ is in general position if and only if $\mathcal{M}(X)\succ0.$
\end{corollary}

We can use the previous theorem to describes $n-$valent vertices in
$\mathbb{CH}^{n}.$ Suppose $\mathbf{V}$ is an $n$-valent vertex in
$\mathbb{CH}^{n}$ with vertex point $x_{1}$ and rays $\left\{  \gamma
_{1i}\right\}  _{i=2}^{n+1}.$ We are only interested in congruence classes and
hence we suppose $x_{1}$ is at the origin. For $i=2,...,n+1$ select a point
$x_{i}$ on $\gamma_{1i}$ and set $X=\left\{  x_{i}\right\}  _{i=1}^{n+1}.$
From the previous theorem we know the congruence class of $X$ is determined by
$\mathcal{M}(X)$ and $\varrho(X).$ Looking at that proof we see that knowing
$\mathcal{M}(X)$ is equivalent to knowing the congruence class of the
projected set $\hat{X}.$ From the definitions we see that knowing $\hat{X}$ is
equivalent to knowing $\mathbf{V.}$ Hence, defining $\mathcal{M}(\mathbf{V})$
to be $\mathcal{M}(X)$ we have the following corollary of the previous theorem:

\begin{corollary}
\label{vertexx}Given an $n$-valent vertex $\mathbf{V}$ in $\mathbb{CH}^{n}$:

\begin{enumerate}
\item $\mathcal{M}(\mathbf{V})$ has $1$'s on the diagonal and $\mathcal{M}%
(\mathbf{V})\succcurlyeq0.$

\item Given $\mathcal{N}$ $\succcurlyeq0$ with $1^{\prime}s$ on the diagonal
there is a $\mathbf{V}$ in $\mathbb{CH}^{n}$ with $\mathcal{M}(\mathbf{V})$
$=\mathcal{N}$.

\item Given $Z=\left\{  z_{i}\right\}  _{i=2}^{n+1}\subset\partial
\mathbb{B}_{n}$ there is a $\mathbf{V}$ in $\mathbb{CH}^{n}$ with
$\mathcal{M}(\mathbf{V})$ equal to the Gram matrix of $Z.$

\item Given other such vertex $\mathbf{W}$, $\mathbf{V}$ is congruent to
$\mathbf{W}$ if and only if $\mathcal{M}(\mathbf{V})=\mathcal{M}(\mathbf{W}).$
\end{enumerate}
\end{corollary}

Neither the definitions of $\varrho(X)$ and $\mathcal{M}(X)$ nor the previous
theorem required that $x_{1}$ is at the origin. However any $X$ is congruent
to a set with $x_{1}$ at the origin and in that case $\mathcal{M}(X)$ is the
Gram matrix of the set $\hat{X},$ and $\delta(x_{1},x_{i})=\left\vert
x_{i}\right\vert .$ In that case $\mathcal{M}(X)$ has the information about
the direction of travel from the origin to the other points of $X$ and
$\rho(X)$ has the information about the distances to be traveled. Using those
descriptions it is clear that the sets of data $\rho$ and $\mathcal{M}$ are
independent of each other, each can be specified freely without regard to the
other. In sum, the set
\[
\mathfrak{S}_{n+1}=\left(  0,1\right)  ^{n}\text{\ }{\huge \times}\text{
}\left\{
\begin{array}
[c]{c}%
\text{positive definite }n\times n\text{ matrices}\\
\text{with ones on the diagonal}%
\end{array}
\right\}  \text{ }%
\]
is a moduli space for the congruence classes of ordered $(n+1)$-tuples in
complex hyperbolic space and the second factor is a moduli space for
$n-$valent vertices. The first factor, $\rho(X),$ consists of $n$ real
numbers. and the second, $\mathcal{M}(X),$ is determined by $n(n-1)/2$ complex
numbers. Together they give the expected total of $n^{2}$ real parameters.

For $n+1=3$ the description is quite simple. From Theorem \ref{2} we know that
the congruence class of the triangle $X=\{x_{1},x_{2},x_{3}\}$ is described by
the set $S^{\prime\prime}=(\delta_{12},\delta_{13},\operatorname*{kos}%
_{1}(2,3)).$ In the notation of the previous theorem%
\[
\varrho(X)=\left(  \delta_{12},\delta_{13}\right)  ,\text{ }\mathcal{M}(X)=%
\begin{pmatrix}
1 & \operatorname*{kos}_{1}(2,3)\\
\operatorname*{kos}_{1}(3,2) & 1
\end{pmatrix}
.
\]
It is also straightforward to describe the parameters for sets contained in
$\mathbb{CH}^{1}=\mathbb{B}_{1},$ the Poincare disk, or in $\mathbb{RH}%
^{2}=BK_{2},$ the Beltrami-Klein model of $\mathbb{RH}^{2}.$ For
$X\subset\mathbb{B}_{1}$ we use the ambient complex coordinate and write
$X=\left\{  r_{s}e^{i\theta_{s}}\right\}  _{s=1}^{n+1}$ with $r_{1}=0.$ Using
the computations in Section \ref{EK} we see that%
\[
\varrho(X)=\left(  r_{2},...,r_{n+1}\right)  ,\text{ }\mathcal{M}(X)=\left(
\exp i\left(  \theta_{s}-\theta_{t}\right)  \right)  _{s,t=2}^{n+1}.
\]
The automorphisms of $\mathbb{CH}^{1}$ which fix the base point are rotations.
They do not change the entries in $\mathcal{M}(X)$ or the congruence class of
$X.$ On the other hand, complex conjugation, which is not in
$\operatorname*{Aut}\mathbb{B}_{1},$ can change the matrix entries and the
congruence class.

For $Y$ $\subset BK_{2}$ we use the polar coordinates of the containing
$\mathbb{R}^{2}.$ We have $Y=\left\{  \left(  r_{s},\theta_{s}\right)
\right\}  _{s=1}^{n+1}$ with $r_{1}=0.$ Now using Section \ref{EK} gives
\[
\varrho(Y)=\left(  r_{2},...,r_{n+1}\right)  ,\text{ }\mathcal{M}(Y)=\left(
\cos\left(  \theta_{s}-\theta_{t}\right)  \right)  _{s,t=2}^{n+1}.
\]
In this case the map $(r,\theta)\rightarrow(r,-\theta),$ which looks like
complex conjugation, is the restriction of an element of $\operatorname*{Aut}%
\mathbb{B}_{2}$ to $BK_{2},$ namely the map $\left(  z,w\right)
\rightarrow(z,-w).$ That map changes the sign of the $\theta^{\prime}s$ but
that does not change $\mathcal{M}(Y)$ or the congruence class of $Y.$

\subsection{Comparison with the McCullough-Quiggin Theorem}

The McCullough-Quiggin theorem as presented in \cite{AM} is for possibly
infinite dimensional RKHSI $H.$ If $H$ is finite dimensional the theorem takes
the following simpler form:

\begin{theorem}
\label{MQ} If $H\in\mathcal{RK}$ then $H\in\mathcal{CPP}$ if and only if%
\[
\operatorname*{KOS}(H,s)\succcurlyeq0,1\leq s\leq\dim H.
\]

\begin{proof}
This follows from the result in \cite{AM} after two observations. First, that
result uses matrices $\operatorname*{MQ}$ as in (\ref{MQx}) while we use
matrices $\operatorname*{KOS}$ from (\ref{KOS}). The matrices are related to
each other through multiplication by diagonal matrices with positive entries.
Hence they have the same positivity properties and the choice between them in
this context is one of convenience. Second, the theorem in \cite{AM} requires
$\operatorname*{KOS}(J,r)\succcurlyeq0$ for each $J$ which is a finite
dimensional regular subspace of $H.$ However if $H$ is finite dimensional we
can apply the hypotheses to the maximal matrices $\operatorname*{KOS}(H,s).$
Every $J$ we need to consider is a regular subspace of $H$ and hence the
matrix $\operatorname*{KOS}(J,r)$ is a principal submatrix of some
$\operatorname*{KOS}(H,s).$ If $\operatorname*{KOS}(H,s)\succcurlyeq0$ that
implies $\operatorname*{KOS}(J,r)\succcurlyeq0$ and hence we do not need a
separate hypothesis for the smaller matrix.
\end{proof}
\end{theorem}

In finite dimensions Corollary \ref{our mq} is an improvement on the
McCullough-Quiggin theorem, it requires a condition on$\ \operatorname*{KOS}%
(H,1)$ not on the full set of $\operatorname*{KOS}(H,s).$ The proof of Theorem
\ref{general} allows the construction of a set $X$ with $DA(X)\sim H$ using
the data in $\operatorname*{KOS}(H,1).$ The fact that $\operatorname*{KOS}%
(H,s)\succcurlyeq0 $ for the other $s$ is then obtained by working with
$DA(X_{R})$ for $X_{R} $ which is a renumbering of $X.$

On the other hand Theorem \ref{general} is more superficial than Theorem
\ref{MQ}. In the proof of Theorem \ref{general} the relationship between
$H\in\mathcal{CPP}$ and $H=DA(X)$ is taken as known, it was imported into the
discussion at Theorem \ref{reduction}. In contrast, the proof of
Theorem\ \ref{MQ} in \cite{AM} is part the substantial work of establishing
that relationship.

\subsection{Some Assembly Required}

In Theorem \ref{2} we\ identified descriptors of the congruence classes of
triangles in $\mathbb{CH}^{n}.$ In Theorem \ref{reduction} we saw that the
geometry of a finite set $X\subset\mathbb{CH}^{n}$ is determined by the
geometry of its included triangles. In Theorem \ref{general} we saw how to
organize that information into a vector $\rho(X)$ and a matrix $\mathcal{M(}%
X\mathcal{)}$ which together describe $X.$ If $Y$ is a subset of $X$ then
$\rho(Y)$ is a subvector of $\rho(X)$ and $\mathcal{M(}Y\mathcal{)}$ is a
submatrix of $\mathcal{M(}X\mathcal{)}$. We now combine those facts with our
results about $X$ to study the relations between the geometry of subsets of
$X$ and properties of submatrices of $\mathcal{M(}X\mathcal{)}$. The vector
$\rho(X)$ has a surprisingly little role in the analysis.

The following classical result will let us relate the fact that $\mathcal{M(}%
X\mathcal{)}\succcurlyeq0$ to the principal submatrices $\mathcal{PS}%
(\mathcal{M(}X\mathcal{))}$ of $\mathcal{M(}X\mathcal{)}$.

\begin{lemma}
[Sylvester's Criterion]\label{sylvester}An $n\times n$ matrix $A$ satisfies
$A\succcurlyeq0$ if and only if $\det A\geq0$ for all $A\in\mathcal{PS}%
{\large (A),}$ that is, if and only if the principal minors of $A$ are
nonnegative. $A\succ0$ if and only if the leading principal minors are positive.
\end{lemma}

To use this with $A=\operatorname*{KOS}(DA(X),1)$ we want information about
its principal submatrices, $\mathcal{PS}(\operatorname*{KOS}(DA(X),1)).$

\begin{lemma}
If $H\in\mathcal{RK}$ then the matrices in $\mathcal{PS}(\operatorname*{KOS}%
(H,1)) $ are the matrices $\operatorname*{KOS}(J,1)$ for $J$ which are regular
subspaces of $H$ which contain the kernel function $k_{1}.$ In particular if
$H=DA(X)$ then they are the matrices $\operatorname*{KOS}(DA(Y),1)$ for $Y$ a
subset of $X$ which contains the distinguished point.
\end{lemma}

We want to know if sets $\left\{  Y_{i}\right\}  $ are congruent in specified
ways to subsets of a larger set $X;$ that is, can the $\left\{  Y_{i}\right\}
$ be \textit{assembled} into $X?$ In the earlier notation the question is:
given $\left\{  Y_{i}\right\}  \rightrightarrows\,??$ is it true that
$\left\{  Y_{i}\right\}  \rightrightarrows X?$ Question 1 of the introduction
is an example.

The previous results give general tools for analyzing such questions; we will
look at three specific variations. Question 1 deals with $3$ point subsets of
4 point sets. We first generalize that to 3 point subsets of $n+1$ point sets
and then to $n$ point inside a set of size $n+1.$ In the final variation we
consider assembling two 4 point sets along 3 point subsets to form a set with
$5$ points. That is an ad hoc example chosen to show ways these ideas play out
in more complicated situations.

The results are in the language of the previous theorem. In each case there
are no nontrivial conditions on $\rho(X)$, the operational conditions concern
determinants of matrices from $\mathcal{PSM}(X).$ In the first case the
conditions involve all of those matrices, in the second case only one, and in
the third, more intricate, case the conclusion involves conditions on 4 of the
15 determinants.

\subsubsection{Variation 1}

We know from Theorem \ref{reduction} that the congruence class of a set $X$ is
determined by the congruence classes of the triangles in $X$ which contain a
specified base point. We now ask if a given set of triangles can be congruent
to those faces of some tetrahedron $X.$ We want to know if we can map
triangles $\left\{  T_{i}\right\}  _{i=1}^{n}$ in $\mathbb{CH}^{n},$
\begin{align*}
T_{j}  &  =\left\{  y_{j1},y_{j,j+1},y_{j,j+2}\right\}  ,\text{ \ }j<n\\
T_{n}  &  =\left\{  y_{n1},y_{n,n+1},y_{n2}\right\}  ,
\end{align*}
into a set $X=\left\{  x_{1},...,x_{n+1}\right\}  $ with each $y_{jk}$ is
mapped to $x_{k}.$ That is, each image has its first vertex at $x_{1}$ and the
images fill $X$ with each segment $x_{1}x_{t}$ in $X$ covered twice. The
coherence conditions are that two triangle sides that cover the same
$x_{1}x_{t}$ must be the same length.

Recall that $\operatorname*{KOS}(\left\{  Y_{i}\right\}  ,1)$ is the same as
$\operatorname*{KOS}(\left\{  DA(Y_{i})\right\}  ,1).$

\begin{theorem}
\label{v1}Given the coherence data $\left\{  T_{i}\right\}  \rightrightarrows
\,??$ just described, the following are equivalent:

\begin{enumerate}
\item $\exists X,$ $\left\{  T_{i}\right\}  \rightrightarrows X,$

\item $\operatorname*{KOS}(\left\{  T_{i}\right\}  ,1)\succcurlyeq0$,

\item $\forall A\in\mathcal{PS(}\operatorname*{KOS}(\left\{  T_{i}\right\}
,1)),$ $\det A\geq0,$

\item $\forall S\subset\left\{  1,...,n\right\}  ,$ $1\in S,$ $\det
\operatorname*{KOS}(\left\{  T_{i}\right\}  _{y\in S},1)\geq0.$
\end{enumerate}
\end{theorem}

\begin{proof}
This is a direct consequence of Theorem \ref{general} and the two previous lemmas.
\end{proof}

Taking note of Corollary \ref{allthesame} we also have the same result for
spaces $J_{i}\in\mathcal{CPP}$ and the question of moving from $\left\{
J_{i}\right\}  \rightrightarrows\,??$ to $\left\{  J_{i}\right\}
\rightrightarrows H$ with $H\in\mathcal{CPP}$.

For both the $\left\{  T_{i}\right\}  $ and the $\left\{  J_{i}\right\}  $ the
condition $\det A\geq0$ is automatic for those $A$ that are $1\times1$ and is
insured by Theorem \ref{2} if $A$ is $2\times2.$ If the $\left\{
J_{i}\right\}  $ are assumed to be in $\mathcal{RK}$ but not necessarily in
$\mathcal{CPP}$ then the situation is more complicated, see the comment after
Theorem \ref{Q 2}.

\subsubsection{Variation 2}

Suppose we are given $\left\{  Y_{i}\right\}  _{i=1}^{n},$ sets of size $n$ in
$\mathbb{CH}^{n}$ \ Write $Y_{i}=\left\{  y_{ij}:1\leq j\leq n+1,\text{ }j\neq
i+1\right\}  .$ We impose the coherence conditions $\left\{  Y_{i}\right\}
\rightrightarrows\,??$ that would hold if there were a set
\[
X=\left\{  x_{1},...,x_{n+1}\right\}
\]
and congruences $Y_{i}\leadsto X$ which mapped the points $y_{is}$ to $x_{s},
$ for all $s\neq i.$ This coherence condition $\left\{  Y_{i}\right\}
\rightrightarrows\,??$ requires strong interrelations between the $Y_{i}$.
Given $Y_{r}$ and $Y_{s}$ there are $Y_{rs}\subset Y_{r}$ and $Y_{sr}\subset
Y_{s}$ both of size $n-1$ with $Y_{rs}\sim Y_{sr}$. Also, note that
every\ $A\in\mathcal{PS(}\operatorname*{KOS}(\left\{  Y_{i}\right\}  ,1))$
which is not maximal, $A\neq\operatorname*{KOS}(\left\{  Y_{i}\right\}  ,1),$
satisfies $A\in\mathcal{PS(}\operatorname*{KOS}(Y_{r},1))$ for some individual
$Y_{r}.$ We know $\operatorname*{KOS}(Y_{r},1)\succcurlyeq0$ and hence, by
Sylvester's criterion $\det A\geq0.$ In sum, the only $A\in\mathcal{PS(}%
\operatorname*{KOS}(\left\{  Y_{i}\right\}  ,1))$ for which we do not know
$\det A\geq0$ is the matrix $\operatorname*{KOS}(\left\{  Y_{i}\right\}  ,1)$
itself. This discussion, together with Theorem \ref{general}, and the previous
two lemmas, complete the proof of the following:

\begin{theorem}
\label{v2}Given $\left\{  Y_{i}\right\}  \rightrightarrows\,??$, there is an
$X$ so that $\left\{  Y_{i}\right\}  \rightrightarrows X$ if and only if
$\det\operatorname*{KOS}(\left\{  Y_{i}\right\}  ,1)\geq0$.
\end{theorem}

\subsubsection{Variation 3}

In the previous two variations the coherence requirements on the $\left\{
Y_{i}\right\}  $ were minimal and maximal. We now look at an intermediate case
which is rich enough to display some structure and simple enough for explicit computations.

Suppose we have two four point sets in $\mathbb{CH}^{n},$ $Y_{A}=\left\{
a_{1},a_{2},a_{3},a_{4}\right\}  $ and $Y_{B}=\left\{  b_{1},b_{3},b_{4}%
,b_{5}\right\}  .$ The coherence requirements, $\left\{  Y_{A},Y_{B}\right\}
\rightrightarrows\,??,$ are the congruences that would hold if we had maps
$Y_{A},Y_{B}\leadsto X=\left(  x_{1},...,x_{5}\right)  $ which respect the
subscripts of the points. If that holds then the triangles $\left\{
a_{1},a_{3},a_{4}\right\}  $ and $\left\{  b_{1},b_{3},b_{4}\right\}  $ are
congruent, and that congruence is the only coherence requirement.

We should not expect to fill the matrix $\operatorname*{KOS}(\left\{
Y_{A},Y_{B}\right\}  ,1).$ The set $X$ has 5 points and so is determined by
$\left(  5-1\right)  ^{2}=16$ real parameters. On the other hand, each of
$Y$'s provides 9 parameters but $4$ of those are pinned by the fact that two
triangles are congruent, leaving 14. This suggests our description is two real
or one complex parameter short of being able to fully describe $X.$ In fact we
cannot construct the entry $\operatorname*{kos}_{1}(2,5)$ in the matrix
$\operatorname*{KOS}(\left\{  Y_{A},Y_{B}\right\}  ,1).$ Doing so would
require a $Y$ which contains points with subscripts $\left\{  1,2,5\right\}
,$ and neither of the $Y^{\prime}s$ satisfy that condition. To move forward we
introduce a new parameter $z$ and fill the matrix $\operatorname*{KOS}%
(\left\{  Y_{A},Y_{B}\right\}  ,1)$ to a matrix $\mathcal{Y}=$
$\operatorname*{KOS}(\left\{  Y_{A},Y_{B},z\right\}  ,1)$ obtained from
$\operatorname*{KOS}(\left\{  Y_{A},Y_{B}\right\}  ,1)$ by putting $z$ in the
place where the $\operatorname*{kos}_{1}(2,5)$ entry would be, and $\bar{z}$
where $\operatorname*{kos}_{1}(5,2)$ would be. The values of $z$ for which
$\operatorname*{KOS}(\left\{  Y_{A},Y_{B},z\right\}  ,1)\succcurlyeq0$, if
any, will parameterize inequivalent possible constructions of the desired $X.$

We need to study the determinants of the matrices in $\mathcal{PS(Y}).$ The
matrix $\mathcal{Y}$ is a $4\times4$ matrix with rows and columns indexed by
the set $\left\{  2,3,4,5\right\}  .$ The matrices in $\mathcal{PS(Y})$ are
determined by the 15 nonempty subsets of that index set. We denote those
matrices by $\mathcal{Y}$ with subscripts denoting the rows, and hence also
columns, of $\mathcal{Y}$ that are retained. There are $4$ single element
subsets to consider, for each of them the resulting matrix has the single
entry $1$ and hence a positive determinant. There are $6$ possibilities with
two subscripts. The matrix $\mathcal{Y}_{34}$ will be a submatrix of both
$\operatorname*{KOS}(Y_{A},1)$ and $\operatorname*{KOS}(Y_{B},1)$ and hence,
by Sylvester's criterion, will have a positive determinant. The matrix
$\mathcal{Y}_{23}$ is not a submatrix of $\operatorname*{KOS}(Y_{B},1),$ but
it is a submatrix of $\operatorname*{KOS}(Y_{A},1)$ and that is enough to
insure it has a positive determinant. The same holds for $\mathcal{Y}_{24}$
and a similar argument applies $\mathcal{Y}_{35}$ and $\mathcal{Y}_{45}$ but
with the roles of $A$ and $B$ reversed. The remaining matrix of that size is
$\mathcal{Y}_{25};$ it cannot be studied using either $Y_{A}$ or $Y_{B}.$ It
is a $2\times2$ matrix with 1's on the diagonal and $z$ and $\bar{z}$ as off
diagonal elements, and it must be dealt with separately. If we do not know
about $\det\mathcal{Y}_{25}$ then we also cannot know the positive
semidefinite nature of the matrices with it as a submatrix; and hence those
matrices must also be studied separately. They are $\mathcal{Y}_{253},$
$\mathcal{Y}_{254},$ and $\mathcal{Y}_{2534}=\mathcal{Y}.$ The two remaining
submatrices are $\mathcal{Y}_{234}=\operatorname*{KOS}(Y_{A},1)$ and
$\mathcal{Y}_{345}=\operatorname*{KOS}(Y_{B},1)$ which we know are positive semidefinite.

\begin{theorem}
\label{v3}There is an $X$ so that $\left\{  Y_{A},Y_{B}\right\}
\rightrightarrows X$ if and only if there is a $z$ so that $\mathcal{Y=}%
\operatorname*{KOS}(\left\{  Y_{A},Y_{B},z\right\}  ,1)\succcurlyeq0,$
equivalently if and only if $\mathcal{Y}$ and the submatrices $\mathcal{Y}%
_{25},$ $\mathcal{Y}_{253},$ and $\mathcal{Y}_{254}$ have positive determinants.
\end{theorem}

\section{Tetrahedra\label{43}}

In the previous section we considered sets of $n+1$ points. For $n=2$ those
results give the equivalence of Conditions (1), (5), and (6) of Theorem
\ref{2}. Now we consider $n=3$. We will make the previous conditions more
explicit and give answers to Questions 1 and 2 of the introduction. We also
use the results to analyze a family of four dimensional RKHS introduced by Quiggin.

\subsection{Question 1\label{question 1}}

We want to know if a set of four triangles in $\mathbb{CH}^{n},$ $\left\{
T_{i}\right\}  _{1}^{4},$ might be congruent to the four faces of a
tetrahedron $X=\left\{  x_{i}\right\}  _{1}^{4}\subset\mathbb{CH}^{3}.$
Actually it suffices to consider only three triangles and we begin by
presenting the argument for that.

\begin{lemma}
If $\left\{  T_{i}\right\}  _{i=1}^{4}\rightrightarrows X$ then the congruence
class of $X$ determines and is determined by the congruence classes of the
$\left\{  T_{i}\right\}  _{i=1}^{4}.$
\end{lemma}

\begin{proposition}
\begin{proof}
That the congruence class of the faces determines the congruence class of the
tetrahedron is Condition (5) of Theorem \ref{reduction}. In the other
direction if we have the coordinates of $X$ then we can read off the
coordinate description of the triangular faces. If we only know the
description of $X$ as given in Theorem \ref{general} -- the designation of a
distinguished vertex $x_{1}$,\ the vector $\rho(X),$ and the matrix
$\mathcal{M}(X)$ -- then we can read off the data sets $S^{\prime}$ for the
three triangular faces of $X$ which meet at $x_{1}.$ The needed side lengths
are entries of $\rho(X)$ and the values of kos at the bivalent vertices at
$x_{1}$ are entries of $\mathcal{M}(X).$ By Theorem \ref{triangle} those data
sets determine the congruence class of the three triangles. To find the
congruence class of the fourth triangle we first note coherence properties the
set of four triangular faces must satisfy; various pairs of sides must have
the same lengths, and the angular invariants must satisfy the cocycle
condition (\ref{kkocycle}). The proof is completed by the next lemma which
shows that that data also suffices to determines the congruence class of the
fourth face.
\end{proof}
\end{proposition}

\begin{proposition}
\label{onlythree}Suppose $\left\{  T_{i}\right\}  _{1}^{4}$ is a set of four
triangles in $\mathbb{CH}^{n}$ which satisfy the matching side length
conditions and the cocycle condition (\ref{kkocycle}) then
\end{proposition}

\begin{lemma}
\begin{enumerate}
\item The congruence classes of $\left\{  T_{i}\right\}  _{1}^{3}$ determines
the congruence class of $T_{4}$.

\item The congruence classes of three triangular faces $\left\{
F_{i}\right\}  _{1}^{3}$ of a tetrahedron determines the congruence class of
the fourth face $F_{4}.$

\item There is a tetrahedron $R_{4}$ with $\left\{  T_{i}\right\}  _{1}%
^{4}\rightrightarrows R_{4}$ if and only if there is a tetrahedron $R_{3} $
with $\left\{  T_{i}\right\}  _{1}^{3}\rightrightarrows R_{3}.$
\end{enumerate}
\end{lemma}

\begin{proof}
For the first statement, if we know the side lengths of $\left\{
T_{i}\right\}  _{1}^{3}$ then using the matching side length condition we also
know the side lengths of $T_{4}$. By the cocycle condition (\ref{kkocycle})
the values of the angular invariants for the $\left\{  T_{i}\right\}  _{1}%
^{3}$ determines the angular invariant for $T_{4}.$ Hence by the results in
Theorem \ref{triangle} for the data set $S^{\prime}$ the congruence class of
$T_{4}$ is determined.

The second statement follows because the collection of faces of a tetrahedron
automatically satisfy the matching side length conditions and the cocycle condition.

One half of the third statement is automatic. In the other direction, suppose
we are given $\left\{  T_{i}\right\}  _{i=1}^{4}$ and know$\left\{
T_{i}\right\}  _{i=1}^{3}\rightrightarrows R_{3}$ for a tetrahedron $R_{3}.$
Denote the faces of $R_{3}$ by $\left\{  F_{i}\right\}  _{i=1}^{4}$ with the
first three congruent to the $\left\{  T_{i}\right\}  _{i=1}^{3}.$ By the
second statement the congruence class of $F_{4}$ is determined by the
$\left\{  F_{i}\right\}  _{i=1}^{3}$. By the first statement the congruence
class of $T_{4}$ is determined by the $\left\{  T_{i}\right\}  _{i=1}^{3}.$
Furthermore, those two determinations use the same argument, in one case with
data from the $\left\{  T_{i}\right\}  $ in the other case data from the
$\left\{  F_{i}\right\}  .$ But the $\left\{  T_{i}\right\}  $ and the
$\left\{  F_{i}\right\}  $ are congruent so those two data sates are the same.
Hence $T_{4}\sim F_{4}$ and hence $R_{3}$ is the required $R_{4}.$
\end{proof}

We now make two reductions which simplify Question 1. First, taking note of
the previous lemma, we can replace $\left\{  T_{i}\right\}  _{1}^{4}$ which
satisfy a matching side length condition and a cocycle condition with a set of
three triangles $\left\{  T_{2},T_{3},T_{4}\right\}  $ which satisfy a
matching side length condition. Second, by Proposition \ref{vertex}, knowing
if those three triangles can be assembled as faces of a tetrahedron is
independent of knowing the side lengths and is determined by the values of
$\operatorname*{kos}$ at the distinguished vertices.

Thus we are reduced to a situation of being told $\left\{  T_{i}\right\}
_{i=1}^{3}\rightrightarrows\,??$ and asking if there is an tetrahedron $X$
with $\left\{  T_{i}\right\}  _{i=1}^{3}\rightrightarrows X$. That notation
includes the assumption that the congruence relations between the vertices of
the triangles, or between the subspaces of the associated $DA$ spaces, have
been specified. We will describe those relations in detail for this particular question.

We start with three triangles $\left\{  T_{i}\right\}  _{i=2}^{4}$ in
$\mathbb{CH}^{k}$ and we denote their coordinates by $T_{i}=$ $\left\{
t_{i1},t_{ia},t_{ib}\right\}  .$ We want to know if we can assemble the
$\left\{  T_{i}\right\}  _{i=2}^{4}$ into a tetrahedron $X=\left\{
x_{i}\right\}  _{i=1}^{4}$ using congruences which take $\left\{
t_{21},t_{2a},t_{2b}\right\}  $ to the triangular face $\left\{  x_{1}%
,x_{2},x_{3}\right\}  ,$ $\left\{  t_{31},t_{3a},t_{3b}\right\}  $ to the
triangular face $\left\{  x_{1},x_{3},x_{4}\right\}  ,$ and $\left\{
t_{41},t_{4a},t_{4b}\right\}  $ to the triangular face $\left\{  x_{1}%
,x_{4},x_{2}\right\}  .$ For that to happen certain points must be identified,
for instance $t_{21},$ $t_{31},$ and $t_{41}$ would be identified at the point
$x_{1}.$ Also certain side lengths will match, for instance because two
triangle sides will coincide as the edge $x_{1}x_{2}$ of $X$ we must have
$\delta(t_{21},t_{2a})=\delta(t_{41},t_{4b}).$

If there is no $X$ then the $\left\{  x_{i}\right\}  $ are only bookkeeping
symbols but the same considerations lead to the same coherence rules.

We now use Theorem \ref{general} to see that $\left\{  T_{i}\right\}  _{2}%
^{4}\rightrightarrows X$ $\ $if and only if $\mathcal{M}\left(  \left\{
T_{i}\right\}  \right)  \succcurlyeq0$. First, suppose $\mathcal{M}\left(
\left\{  T_{i}\right\}  \right)  \succcurlyeq0.$ The side length data from the
$\left\{  T_{i}\right\}  $ is enough to construct $\rho^{\ast}$ the candidate
for $\rho(X).$ By the second statement of that theorem, given that
$\mathcal{M}(\left\{  T_{i}\right\}  )$ $\succcurlyeq0$ then there is an $X$
with $\mathcal{M}\left(  \left\{  T_{i}\right\}  \right)  =\mathcal{M(}%
X\mathcal{)}$ and $\rho^{\ast}=\rho(X).$ From $\mathcal{M}\left(  \left\{
T_{i}\right\}  \right)  =\mathcal{M(}X\mathcal{)}$ and $\rho^{\ast}=\rho(X)$
and the condition in Theorem \ref{reduction} on the data set $S^{\prime\prime
}$ we see that the faces of $X$ are congruent to the $\left\{  T_{i}\right\}
$ and thus $\left\{  T_{i}\right\}  _{2}^{4}\rightrightarrows X.$ In the other
direction, if we knew $\left\{  T_{i}\right\}  _{2}^{4}\rightrightarrows X$ we
would know $\mathcal{M}\left(  \left\{  T_{i}\right\}  \right)  =\mathcal{M(}%
X\mathcal{)}$ and hence $\mathcal{M}\left(  \left\{  T_{i}\right\}  \right)
\succcurlyeq0$.

The entries of $\mathcal{M(}X)$ would be values of $K_{ij}=\operatorname*{kos}%
(\mathbf{V}_{ij})$ for the bivalent vertices of $X$ at $x_{1}.$ If $\left\{
T_{i}\right\}  _{2}^{4}\rightrightarrows X$ then those vertices are congruent
to vertices of the $\left\{  T_{i}\right\}  _{2}^{4}.$ Specifically, writing
$\operatorname*{kos}(T_{i})$ for the value of $\operatorname*{kos}$ at the
vertex $t_{i1}$ in the triangle $T_{i}$ we would have $K_{23}%
=\operatorname*{kos}(T_{2}),$ $\ K_{34}=\operatorname*{kos}(T_{3}),$ and
$K_{24}=\overline{\operatorname*{kos}(T_{4})}.$ ( The final complex
conjugation because $T_{4}$ was placed "backwards" in terms of our numbering
of the edges.). Thus $\mathcal{M(}\left\{  T_{i}\right\}  \mathcal{)}$ is
given by%
\begin{equation}
\mathcal{M(}\left\{  T_{i}\right\}  \mathcal{)}=%
\begin{pmatrix}
1 & \operatorname*{kos}(T_{2}) & \overline{\operatorname*{kos}(T_{4})}\\
\overline{\operatorname*{kos}(T_{2})} & 1 & \operatorname*{kos}(T_{3})\\
\operatorname*{kos}(T_{4}) & \overline{\operatorname*{kos}(T_{3})} & 1
\end{pmatrix}
. \label{mt}%
\end{equation}
We will be interested in positivity properties of matrices of that form.

\begin{lemma}
\label{matrix condition}Suppose we have $\left\{  a_{i}\right\}  _{i=1}%
^{3}\subset\mathbb{C}$ and%
\[
\mathcal{N}=%
\begin{pmatrix}
1 & a_{1} & a_{2}\\
\overline{a_{1}} & 1 & a_{3}\\
\overline{a_{2}} & \overline{a_{3}} & 1
\end{pmatrix}
.
\]
Suppose that for some $i$ $\left\vert a_{i}\right\vert <1$ or that for all $i
$ $\left\vert a_{i}\right\vert \leq1.$ The following are equivalent:

\begin{enumerate}
\item $0\preccurlyeq\mathcal{N}$,

\item $0\leq\det\mathcal{N},$

\item $0\leq1+2\operatorname{Re}a_{1}\overline{a_{2}}a_{3}-\left\vert
a_{1}\right\vert ^{2}-\left\vert a_{2}\right\vert ^{2}-\left\vert
a_{3}\right\vert ^{2},$

\item $\left\vert a_{1}\overline{a_{2}}-a_{3}\right\vert ^{2}\leq(1-\left\vert
a_{1}\right\vert ^{2})(1-\left\vert a_{2}\right\vert ^{2}).$
\end{enumerate}

\begin{proof}
By Sylvester's criterion the first condition implies the second. The second,
third conditions are equivalently definition. That the third and fourth are
equivalent can be seen by expanding both sides of the fourth statement giving%
\[
|a_{1}\overline{a_{2}}|^{2}-2\operatorname{Re}a_{1}\overline{a_{2}}%
a_{3}+\left\vert a_{3}\right\vert ^{2}\leq1-\left\vert a_{1}\right\vert
^{2}-\left\vert a_{2}\right\vert ^{2}+\left\vert a_{1}a_{2}\right\vert ^{2}.
\]
Cancellation and rearrangement shows that is equivalent to the third statement.

To go back to the first statement we want to use Sylvester's criterion. That
states that the first statement is a consequence of the nonnegativity of the
seven principal minors. Three are the determinants of the $1\times1$ matrices
given by the diagonal entries and they are positive, One is the positivity of
$\det\mathcal{N}$ which is insured by the second statement. The other three
requirements are equivalent to the conditions that each $\left\vert
a_{i}\right\vert \leq1$ and hence if we had assumed the $\left\vert
a_{i}\right\vert \leq1$ then we would be done. Otherwise note that the right
hand side of the third statement is nonnegative. That lets us see that if, for
instance, $\left\vert a_{1}\right\vert <1$ then $\left\vert a_{2}\right\vert
\leq1$ and a similar argument applies to $a_{3}.$ Thus we have reduced to the
case of all $\left\vert a_{i}\right\vert \leq1.$
\end{proof}
\end{lemma}

We have collected all of the pieces to answer Question 1.

\begin{theorem}
\label{Q 1}With the numbering and naming scheme just described

\begin{enumerate}
\item If $\left\{  T_{i}\right\}  \rightrightarrows X$ then $\mathcal{M}%
(\left\{  T_{i}\right\}  )=\mathcal{M}(X)\succcurlyeq0$.

\item If $\mathcal{M}(\left\{  T_{i}\right\}  )$ $\succcurlyeq0$ then there is
an $X$ with $\left\{  T_{i}\right\}  \rightrightarrows X$ in which case
$\mathcal{M}(\left\{  T_{i}\right\}  )=\mathcal{M}(X).$

\item The previous lemma applies to the statements $\mathcal{M}(\left\{
T_{i}\right\}  )\succcurlyeq0$ and $\mathcal{M}(X)\succcurlyeq0.$ In
particular $\mathcal{M}(X)\succcurlyeq0$ if and only if
\begin{equation}
\left\vert K_{34}-\overline{K_{23}}K_{24}\right\vert ^{2}\leq(1-\left\vert
K_{23}\right\vert ^{2})(1-\left\vert K_{24}\right\vert ^{2}). \label{hopw}%
\end{equation}

\end{enumerate}
\end{theorem}

\begin{proof}
The first two statements follow from the discussion before the theorem
together with Theorem \ref{general}. The third statement follows from the
previous lemma together with statement (6) of Theorem \ref{2} which insures
that each $\left\vert K_{ij}\right\vert \leq1.$
\end{proof}

The condition (\ref{hopw}) was obtained by specializing general results. It is
interesting to also see how it follows from working directly with coordinates.
We can assume $X=\Delta$ as described in (\ref{tetra}) in which case the
$K_{ij}$ can be computed using the points $\{\widehat{x_{i}}\}_{i=2}^{4}$ on
the unit sphere with coordinates%

\begin{align}
&  \widehat{x}_{2}=(1,0,0),\text{ }\widehat{x}_{3}=(\xi,\beta,0),\text{
}\widehat{x}_{4}=(\eta,\zeta,\gamma);\nonumber\\
&  \beta,\gamma\geq0;\text{ }\xi,\eta,\zeta\in\mathbb{C};\text{ }\nonumber\\
&  \left\vert \xi\right\vert ^{2}+\beta^{2}=\left\vert \eta\right\vert
^{2}+\left\vert \zeta\right\vert ^{2}+\gamma^{2}=1. \label{gamma}%
\end{align}

\noindent For $2\leq i,j\leq4$ we have $K_{ii}=1,$ $K_{ij}=\overline{K_{ji}}.$
The rest of the story is given by%
\[
K_{23}=\bar{\xi},\text{ \ }K_{24}=\bar{\eta},\text{ \ }K_{34}=\xi\bar{\eta
}+\beta\bar{\zeta}.
\]
We also have
\[
\left\vert \beta\right\vert ^{2}=1-\left\vert K_{23}\right\vert ^{2};\text{
}\left\vert \bar{\zeta}\right\vert ^{2}=1-\left\vert K_{24}\right\vert
^{2}-\gamma^{2}.
\]
Hence, noting that for all $i,j,$ $\left\vert K_{ij}\right\vert \leq1,$ we
must have%
\begin{equation}
\left\vert K_{34}-\overline{K_{23}}K_{24}\right\vert ^{2}\leq(1-\left\vert
K_{23}\right\vert ^{2})(1-\left\vert K_{24}\right\vert ^{2}). \label{cauchy 2}%
\end{equation}
which is (\ref{hopw}).

In the other direction it is not hard to start from $\left\{  K_{ij}\right\}
$ which satisfy these conditions and find coordinates of points on the sphere
which generate this data.

The first two statements in Theorem \ref{Q 1} are special cases of both
Theorem \ref{v1} and Theorem \ref{v2}. The next result can be seen as a
simpler variant of Theorem \ref{v3}.

If $T_{2}$ and $T_{3}$ of $\left\{  T_{i}\right\}  _{i=2}^{4}$ are specified
then we can use the previous result to give conditions on $T_{4}$ which insure
that $\left\{  T_{i}\right\}  _{i=2}^{4}\rightrightarrows X$. Specifically if
we are given a $T_{2}$ and $T_{3}$ and the coherence conditions then the side
lengths of $T_{2}$ and $T_{3}$ are enough to fully specify $\rho(X).$ The
value of $\operatorname*{kos}$ at the distinguished vertex of $T_{2}$ will
give the value of $K_{23}$ for the matrix $\mathcal{M(}\left\{  T_{2}%
,T_{3}\right\}  ).$ Similarly for $T_{3}$ and $K_{24}.$ For us to assemble the
triangles into the tetrahedron the lengths of the sides of $T_{4}$ must match
appropriately lengths of sides of $T_{2}$ and $T_{3}.$ However the data from
$T_{2}$ and $T_{3}$ is not enough to compute $K_{42.}$ which would be the
value of $\operatorname*{kos}$ at its distinguished vertex of $T_{4}$. Hence
the congruence class of $T_{4}$ is indeterminate. To go forward we set
$K_{42.}=z$ and using $z$ we complete the matrix $\mathcal{M(}\left\{
T_{2},T_{3}\right\}  )$ to a matrix $\mathcal{M(}\left\{  T_{2},T_{3}\right\}
,z)$ which has no missing values. We can then apply the previous theorem to
that matrix.

\begin{corollary}
If $\left\{  T_{2},T_{3}\right\}  \rightrightarrows\,??$ then there is a third
triangle $T_{4}$ and a tetrahedron $X$ with $\left\{  T_{2},T_{3}%
,T_{4}\right\}  \rightrightarrows X$ if and only if the Euclidean ball in
$\mathbb{C}^{1}$%
\[
B=B\left(  \overline{K_{23}}K_{24},(1-\left\vert K_{23}\right\vert ^{2}%
)^{1/2}(1-\left\vert K_{24}\right\vert ^{2})^{1/2}\right)
\]
is nonempty. In that case the pairing of $z$ with the value $K_{42}$
establishes a one to one correspondence between $z\in B$ and the congruence
class of the possible third triangle $T_{4}.$ If $B$ is empty then there are
no such $T_{4}.$

\begin{proof}
Putting $z$ into (\ref{hopw}) gives $\left\vert z-\overline{K_{23}}%
K_{24}\right\vert ^{2}\leq(1-\left\vert K_{23}\right\vert ^{2})(1-\left\vert
K_{24}\right\vert ^{2}).$
\end{proof}
\end{corollary}

\subsection{Question 2\label{question 2}}

We saw in Section \ref{equivalence} that Question 2 is equivalent to Question
1, Having answered Question 1 we now reconfigure that answer to apply to
Question 2. We will be informal.

We start with three dimensional $\left\{  J_{i}\right\}  _{i=1}^{4}%
\subset\mathcal{CPP}.$ We assume the matching side length conditions and the
cocycle condition, now defined using the invariants $\delta$ and $\alpha$
defined directly from reproducing kernels of the $\left\{  J_{i}\right\}
_{i=1}^{4}$. We want to know if there is a four dimensional $H\in
\mathcal{CPP}$ whose four three dimensional regular subspaces are rescalings
of the $\left\{  J_{i}\right\}  .$ As in the previous section it suffices to
consider only the three spaces $\left\{  J_{i}\right\}  _{i=2}^{4}$. Given
those spaces we can construct the matrix $\mathcal{M}(\left\{  J_{i}\right\}
),$ the matrix that will equal $\operatorname*{KOS}(H,1)$ if there is an $H.$
The entries of $\mathcal{M}(\left\{  J_{i}\right\}  )$ are the values
$\left\{  L_{ij}\right\}  $ of $\operatorname*{kos}$ for the $\left\{
J_{i}\right\}  ,$ we obtain%
\begin{equation}
\mathcal{M}(\left\{  J_{i}\right\}  )=%
\begin{pmatrix}
1 & L_{23} & L_{24}\\
L_{32} & 1 & L_{34}\\
L_{42} & L_{43} & 1
\end{pmatrix}
. \label{ell}%
\end{equation}

\begin{theorem}
\label{Q 2}Given $\left\{  J_{i}\right\}  _{i=2}^{4}\subset\mathcal{CPP}$ and
$\left\{  J_{i}\right\}  _{i=1}^{4}\rightrightarrows\,??$ there is a four
dimensional $H\in\mathcal{CPP}$ with $\left\{  J_{i}\right\}  _{i=2}%
^{4}\rightrightarrows\,H$ if and only if $\mathcal{M}(\left\{  J_{i}\right\}
)\succcurlyeq0.$
\end{theorem}

Because $\left\{  J_{i}\right\}  _{i=2}^{4}\subset\mathcal{CPP}$ we know from
Theorem \ref{2} that each $\left\vert L_{ij}\right\vert \leq1.$ Hence Lemma
\ref{matrix condition} can be applied and that gives several conditions
equivalent to $\mathcal{M}(\left\{  J_{i}\right\}  )\succcurlyeq0$ including
$\det\mathcal{M}(\left\{  J_{i}\right\}  )\geq0.$ However if we only know that
$\left\{  J_{i}\right\}  _{i=2}^{4}\subset\mathcal{RK}$ then deriving
$\mathcal{M}(\left\{  J_{i}\right\}  )\succcurlyeq0$ from $\det\mathcal{M}%
(\left\{  J_{i}\right\}  )\geq0$ requires the additional assumption that the
$\left\vert L_{ij}\right\vert \leq1.$ (However this is not actually a
different formulation.. By Theorem \ref{2}, adding the assumptions that
$\left\vert L_{ij}\right\vert \leq1$ is equivalent to passing from the
assumption that $\left\{  J_{i}\right\}  _{i=2}^{4}\subset\mathcal{RK}$ to the
assumption that $\left\{  J_{i}\right\}  _{i=2}^{4}\subset\mathcal{CPP}.)$

\subsection{Quiggin's Example\label{quiggen}}

If a three dimensional RKHSI $H$ has the Pick property then it has the
complete Pick property. This can be seen from the implication
$(4)\Longrightarrow\left(  2\right)  $ in Theorem \ref{triangle} and is not
hard to prove directly. It is then natural to ask if a four dimensional $H$
with the Pick proper must have the $CPP.$ It was shown by Quiggin in his
thesis that this is not true, \cite{Q}, \cite[Page 94]{AM}. He produced a
family of four dimensional RKHSI $H_{x},$ $0<x<1,$ with the Pick property (and
hence whose three dimensional regular subspaces each have the Pick property,
and hence also the CPP) and showed that the space $H_{1/4}$ failed the CPP.

Here we consider the spaces $\left\{  H_{x}\right\}  $ using the results from
the previous sections. We do not verify that the $\left\{  H_{x}\right\}  $
have the Pick property (that is done in \cite[pg, 84]{Q}), but we will show
that for each $H_{x}$ the regular three dimensional subspaces have the CPP. We
also show that none of the $H_{x}$ have the CPP.

Following Quiggin we introduce a family $\left\{  H_{x}:0<x<1\right\}
\subset\mathcal{RK}$ of four dimensional spaces by specifying their Gram
matrices, $\operatorname*{Gr}(H_{x}).$ For $0<x<1$ and $s=(1-x)\sqrt{x}$ set%
\[
\operatorname*{Gr}(H_{x})=%
\begin{pmatrix}
1 & x & x & x+is\\
x & 1 & x-is & x\\
x & x+is & 1 & x\\
x-is & x & x & 1
\end{pmatrix}
.
\]
To show this is the Gram matrix of a RKHSI we need to show that
$\operatorname*{Gr}(H_{x})\succ0.$ By Lemma \ref{sylvester} we can do that by
checking the signs of the leading principal minors. They are
\[
\left(  1+x\right)  ^{2}\left(  1-x\right)  ^{4},\left(  1+x\right)  \left(
1-x\right)  ^{2},\left(  1+x\right)  \left(  1-x\right)  ,1
\]
and, by inspection, are all positive for $0<x<1$. (Those computations and the
determinant computations below were done using computer algebra.)

Earlier we used the matrices $\operatorname*{KOS}(\operatorname*{Gr}%
(H_{x}),1)$ from (\ref{ell}). Here for ease in computing we use the matrices
$\operatorname*{MQ}(H_{x},1)=$ $\left(  \delta_{1i}\delta_{1j}%
\operatorname*{kos}\nolimits_{1}(i,j)\right)  _{i,j=2}^{4}$ mentioned in
(\ref{MQx}). The two have determinants of the same sign as do their square submatrices.

From the definitions we have
\[
\operatorname*{MQ}(H_{x},1)=%
\begin{pmatrix}
1-x^{2} & 1-\frac{x^{2}}{x-is} & 1-x-is\\
1-\frac{x^{2}}{x+is} & 1-x^{2} & 1-x-is\\
1-x+is & 1-x+is & \left(  1-x\right)  (1+x^{2})
\end{pmatrix}
\]
{}

Fix $x.$ We want to know that $\left\{  J_{xi}\right\}  _{1}^{4},$ the regular
three dimensional subspaces of $H_{x},$ have the CPP. By (5) of Theorem
\ref{triangle} we know that $J_{x2}\in\mathcal{CPP}$ if the matrix
$\mathcal{J}_{2}$ obtained by deleting the first row and first column of
$\operatorname*{MQ}(H_{x},1)$ satisfies $\mathcal{J}_{2}\succcurlyeq0$. That
will follow if we show $\det\mathcal{J}_{2}\geq0.$ Similarly for $J_{x3}$ and
$J_{x4}$. For $J_{x1}$ we follow the same path but starting with
$\operatorname*{MQ}(H_{x},2)$ rather than $\operatorname*{MQ}(H_{x},1).$ To
show that $H_{x}\notin\mathcal{CPP}$ we will show that $\det\operatorname*{MQ}%
(H_{x},1)<0$ and hence $\operatorname*{MQ}(H_{x},1)\succcurlyeq0$ fails. All
these things can be seen in the explicit formulas for the determinants. Note
that for $0<x<1,$ we have $x^{2}-x+1>0.$ We have
\begin{align*}
\det\mathcal{J}_{1}  &  =\det\mathcal{J}_{2}=\det\mathcal{J}_{3}=x^{2}\left(
x+1\right)  \left(  x-1\right)  ^{2},\\
\det\mathcal{J}_{4}  &  =\frac{x^{3}\left(  x+1\right)  \left(  x-1\right)
^{2}}{x^{2}-x+1},\\
\det\operatorname*{MQ}(H_{x},1)  &  =-\frac{x^{3}\left(  x+1\right)
^{2}\left(  x-1\right)  ^{4}}{x^{2}-x+1}.
\end{align*}

This shows that the matching distances property together with the cocycle
property are not sufficient to insure that a set of four three dimensional
spaces with the CPP can be assembled into a four dimensional space with the CPP.

\begin{proposition}
Fix $x,$ $0<x<1,$ let $\left\{  J_{xi}\right\}  _{i=1}^{4}$ be the three
dimensional regular subspaces of $H_{x}$. Then

\begin{enumerate}
\item The $\left\{  J_{i}\right\}  _{i=1}^{4}\subset\mathcal{CPP}$ and
$\left\{  J_{i}\right\}  _{i=1}^{4}\rightrightarrows\,??$.

\item There is an $H\in\mathcal{RK}$ with $\left\{  J_{i}\right\}  _{i=1}%
^{4}\rightrightarrows H$.

\item There is no $H\in\mathcal{CPP}$ with $\left\{  J_{i}\right\}  _{i=1}%
^{4}\rightrightarrows H$.
\end{enumerate}
\end{proposition}

\begin{proof}
We verified above that the $\left\{  J_{i}\right\}  $ all have the CPP. The
coherence is automatic because the $\left\{  J_{i}\right\}  $ are the three
dimensional regular subspaces of a four dimensional RKHSI. The second
statement is tautological, $H=H_{x}$ will suffice. It is included to emphasize
that there is no obstruction to assembling the $\left\{  J_{i}\right\}  $ into
a containing RKHSI $H,$ just not one with the CPP.
\end{proof}

Because $\left\{  J_{xi}\right\}  \subset\mathcal{CPP}$ there are triangles
$\left\{  T_{xi}\right\}  $ in $\mathbb{CH}^{n}$ with $J_{xi}\sim DA(T_{xi}).$
The coherence of the $\left\{  J_{xi}\right\}  $ insures that the $\left\{
T_{xi}\right\}  $ satisfy the coherence conditions for assembly as a
tetrahedron. Hence the previous result can be recast as a result about the
possible assembly of those triangles.

\subsection{Tetrahedra in $\mathbb{RH}^{k}$}

The study of polyhedra in $\mathbb{RH}^{k}$ is a very rich topic which is
considered in many places; for instance the books \cite{Co}, \cite{F},
\cite{An}, surveys \cite{J}, \cite{MP}, and research papers \cite{HR},
\cite{Di} \cite{W}. The relation between HSRK and sets in $\mathbb{RH}^{k}$ is
studied in \cite{BIM}, \cite{M} and \cite[Sec. 7]{Ro}. Here we look at what
happens when the results related to Theorem \ref{general} are specialized to
tetrahedra in $\mathbb{RH}^{k}.$

We will call sets in $\mathbb{CH}^{n}$ that are totally geodesic submanifolds
isometric to $\mathbb{RH}^{k}$ \textit{copies of} $\mathbb{RH}^{k}.$ A four
point set in a copy of $\mathbb{RH}^{k}$ is a real hyperbolic tetrahedron. For
those sets the value of $\operatorname*{kos}$ at a bivalent vertex specializes
as the the cosine of the vertex angle. This leads to simplifications of some
of the previous results and to new geometric questions.

There are two technical points to discuss before going forward. First, there
is an ambiguity in saying that two such sets in a copy $\mathbb{RH}^{k}$
inside of $\mathbb{CH}^{n}$ are congruent. We could mean that there is an
automorphism of $\mathbb{CH}^{n}$ which takes one to the other. Alternatively
we might mean that there is an automorphism of $\mathbb{RH}^{k} $ which takes
one set to the other, with no mention of an automorphism of $\mathbb{CH}^{n}.$
Fortunately these two notions are equivalent \cite{Go}. A similar comment
holds for sets inside two different copies of $\mathbb{RH}^{k}$ inside
different copies $\mathbb{CH}^{n}.$

Second, we have been studying $X$ using invariants derived from $DA(X$) but
have only defined $DA(X)\ $for $X$ inside a copy of $\mathbb{RH}^{k}$ inside
some $\mathbb{CH}^{n}.$ To have definitions that apply to an $X$ in any
version of $\mathbb{RH}^{k}$ we select $\Lambda,$ an isometric map of that
$\mathbb{RH}^{k}$ onto a copy of $\mathbb{RH}^{k}$ inside a $\mathbb{CH}^{n},$
and then define $DA(X)$ to be $DA(\Lambda(X)).$ The construction depends on
the choice of $\Lambda$ but changing $\Lambda$ produces a rescaling of $DA(X)$
and that does not change the values of the invariants we work with.

\subsection{Sets Inside Copies of $\mathbb{RH}^{k}$\label{rhk}}

If $S$ is a copy of $\mathbb{RH}^{k}$ with $k=1$ then $S$ is an ordinary
geodesic, if $k=2$ then, recalling $BK_{2}$ from Section \ref{hg},
$S=\phi(BK_{2})$ for some $\phi\in$ $\operatorname*{Aut}\mathbb{B}_{n}.$ The
cases $k>2$ follow a similar pattern, those sets are automorphic images of the
set $BK_{k}$ obtained by intersecting $\mathbb{B}_{n}$ with a copy of
$\mathbb{R}^{k}$ which is inside $\mathbb{C}^{n}.$ These ideas are discussed
in more detail in \cite{Go} and \cite{BI}. The fact that $X\subset
\mathbb{CH}^{n}$ is inside such an $S$ is invariant under automorphisms of the
ambient $\mathbb{CH}^{n}.$ There are various intrinsic characterizations of
sets $X$ with this property; the following proposition is consequence of
\cite[Lemma 2.1]{BI}.

\begin{proposition}
\label{real}$\left\{  x_{i}\right\}  _{i=1}^{s}\subset\mathbb{CH}^{n}$ is
inside a copy of $\mathbb{RH}^{k}$ if and only if all the numbers
$\operatorname*{kos}_{i}(p,q)$ are real.
\end{proposition}

Recall that there is a notion of the hyperbolic angle of intersection for
curves in $\mathbb{RH}^{2}.$

\begin{corollary}
\label{real kos} The triangle $T=\left\{  x_{1},x_{2},x_{3}\right\}
\subset\mathbb{CH}^{n}$ is in a copy of $\mathbb{RH}^{2}$ if and only if
$\operatorname*{kos}_{1}(2,3)$ is real. In that case $\operatorname*{kos}%
_{1}(2,3)=\cos va_{23}$ where $va_{23}$ is hyperbolic angle at the vertex
formed by the hyperbolic geodesics $x_{1}x_{2}$ and $x_{1}x_{3}.$
\end{corollary}

\begin{proof}
Using the model triangle $\Gamma$ in (\ref{tri}) it is easy to check that if
$\operatorname*{kos}_{1}(2,3)$ is real then the coordinates of the points of
$\Gamma$ are real and hence also so are the other values of
$\operatorname*{kos}.$ It then follows from the previous proposition that $T$
is in a copy of $\mathbb{RH}^{k}.$ Because $\Gamma$ only has three points we
can take $k=2.$ Using (\ref{simple kos}) we see that $\operatorname*{kos}%
_{1}(2,3)=\left\langle \left\langle \widehat{x_{2}},\widehat{x_{3}%
}\right\rangle \right\rangle .$ That inner product equals the cosine of the
Euclidean angle at the origin of $\mathbb{R}^{2}$ between the segments
$0\widehat{x_{2}}$ and $0\widehat{x_{3}}$. At the origin the Euclidean metric
on $\mathbb{R}^{2}$ is conformal with the hyperbolic metric and hence the
Euclidean angle whose cosine we found is also the hyperbolic angle.
\end{proof}

\subsubsection{Vertex Angles\label{rht}}

Angles such as $va_{23}$ in the previous corollary are called \textit{vertex
angles}. (On a polyhedron they are also sometimes called face angles or facial
angles.) The previous corollary shows that if $X=\left\{  x_{1},x_{2}%
,x_{3},x_{4}\right\}  $ is in a copy of $\mathbb{RH}^{n}$ then
$\operatorname*{kos}_{i}(j,k)$ is the cosine of the vertex angle at the vertex
$x_{j}x_{1}x_{i}.$ Denoting that angle by $va_{ij}$ the entries of
$\mathcal{M}(X)=$ $\operatorname*{KOS}(DA(X),1)$ are $K_{ij}%
=\operatorname*{kos}\nolimits_{i}(j,k)=\cos va_{jk}$. Theorem \ref{Q 1}
applies to the matrix and the formulas simplify slightly because the $K_{ij}$
are real. To emphasize this change we introduce a modified notation for the
matrix $\operatorname*{KOS}(DA(X),1).$ We set%
\begin{equation}
\mathcal{M}_{CV\!\!A}(X)=\operatorname*{KOS}(DA(X),1)=%
\begin{pmatrix}
1 & \cos va_{23} & \cos va_{24}\\
\cos va_{32} & 1 & \cos va_{34}\\
\cos va_{42} & \cos va_{43} & 1
\end{pmatrix}
. \label{CVA}%
\end{equation}
and note that $\cos va_{ij}=\cos va_{ji}.$ Thus $\mathcal{M}_{CV\!\!A}%
(X)=\mathcal{M}(X),$ the added subscript is a reminder that the matrix entries
are cosines of vertex angles.

If we have triangles $\left\{  T_{i}\right\}  _{i=2}^{4}$ in a copy of
$\mathbb{RH}^{n}$ which satisfy the coherence conditions for assembly into a
tetrahedron, $\left\{  T_{i}\right\}  _{i=2}^{4}\rightrightarrows\,??$ we set
$\mathcal{M}_{CV\!\!A}(\left\{  T_{i}\right\}  _{i=2}^{4})=\mathcal{M}%
(\left\{  T_{i}\right\}  _{i=2}^{4}).$ Thus if $\left\{  T_{i}\right\}
_{i=2}^{4}\rightrightarrows\,X$ then $\mathcal{M}_{CV\!\!A}(\left\{
T_{i}\right\}  _{i=2}^{4})=\mathcal{M}_{CV\!\!A}(X).$ Those entries are real
and hence, by the previous proposition, if there is an $X$ then it is in a
\underline{real} hyperbolic space.

Theorem \ref{Q 1} applies in this context and gives

\begin{corollary}
\label{first real} Given triangles $\left\{  T_{i}\right\}  _{i=1}^{4}$ in
$\mathbb{RH}^{n}$ with $\left\{  T_{i}\right\}  _{i=2}^{4}\rightrightarrows
\,?? $ there is a tetrahedron $X$ in $\mathbb{RH}^{n}$ such that $\left\{
T_{i}\right\}  _{i=2}^{4}\rightrightarrows X$ if and only if $\mathcal{M}%
_{CV\!\!A}(\left\{  T_{i}\right\}  _{i=2}^{4})\succcurlyeq0.$ That condition
is equivalent to each of the following three conditions:
\begin{align}
\det\mathcal{M}_{CV\!\!A}(\left\{  T_{i}\right\}  _{i=2}^{4})  &
\geq0\label{det}\\
(\cos va_{34}-\cos va_{23}\cos va_{42})^{2}  &  \leq(1-\cos^{2}va_{23}%
)(1-\cos^{2}va_{42}),\label{trig va}\\
\left(  \frac{\cos va_{34}-\cos va_{23}\cos va_{42}}{\sin va_{23}\sin va_{42}%
}\right)  ^{2}  &  \leq1, \label{cauchy 3}%
\end{align}

\end{corollary}

\begin{proof}
The conditions $\mathcal{M}_{CV\!\!A}(\left\{  T_{i}\right\}  _{i=2}%
^{4})\succcurlyeq0,$ (\ref{det}), and (\ref{trig va}) are repeats of parts of
Theorem \ref{Q 1}. The final statement is a rewriting of (\ref{trig va}) which
will be useful later.
\end{proof}

Each of these conditions is necessary and sufficient conditions for the vertex
angles of a tetrahedron in $\mathbb{RH}^{2}.$ We will see a much simpler
equivalent condition in Proposition \ref{conditions}.

\subsubsection{Dihedral Angles\label{da}}

Going forward we will suppose that the tetrahedra we consider in
$\mathbb{RH}^{k}$ are nondegenerate, they are not congruent to tetrahedra in
$\mathbb{RH}^{2}.$ It will be convenient now realize $\mathbb{RH}^{3}$ as the
Poincar\'{e} ball model. In that model the manifold is the real three ball,
$\mathbb{RB}_{3},$ the group action is the group of conformal automorphisms of
the ball, and the "sphere at infinity" of $\mathbb{RH}^{3}$ is the Euclidean
two sphere $\mathbb{S}_{2}=\partial\mathbb{RB}_{3}.$ These choices let us use
Euclidean coordinate geometry.

If $X=\left\{  x_{i}\right\}  _{i=1}^{4}$ is a tetrahedron in $\mathbb{RH}%
^{3}$ then we know from Proposition \ref{vertex} that the vertex angles
$\left\{  va_{ij}\right\}  $ describe the congruence class of the trivalent
vertex at $x_{1}.$ That vertex can also be described using the angles between
faces, the \textit{dihedral angles}. In fact the dihedral angles are used more
commonly than vertex angles in describing real hyperbolic polyhedra; see for
instance \cite{W}, \cite{FG}, \cite{HR}, \cite{Roe}, and the references there.
We now look briefly at describing the dihedral angles of a tetrahedron in our
framework and at their relation to the vertex angles.

The work in this section is influenced by the work of Roeder in \cite{Roe} and
there are overlaps. In \cite{Roe} tetrahedra are described using the matrix of
negatives of cosines of dihedral angles. The matrix $\mathcal{M}_{-CDA}(X)$ we
introduce below is a submatrix of that matrix and some of the results about
submatrices in Theorem 1 in \cite{Roe} match some of the results in Theorem
\ref{Q 1 DA} below. There are also differences. In particular the work in
\cite{Roe} restricts attention to tetrahedra with non-obtuse dihedral angles.

In the next few paragraphs we will use the indices $r,s,t$ to denote three
different indices from the set $\left\{  2,3,4\right\}  .$ Given the
tetrahedron $X$ in $\mathbb{RH}^{3}$ with $x_{1}$ at the origin denote the
triangular face with vertices $\left\{  x_{1},x_{i},x_{j}\right\}  $ by
$F_{ij}.$ The dihedral angle along edge $s,$ $da_{s},$ $s=2,3,4,$ is the angle
between the faces $F_{rs}$ and $F_{st}.$ If this were a Euclidean tetrahedron
we could find $da_{s}$ by finding the inward pointing unit normals, $n_{rs}$
for the face $F_{rs},$ and using the fact that $-\cos da_{s}=\left\langle
n_{rs},n_{st}\right\rangle .$ The requirement that $n_{rs}$ be inward pointing
is the requirement that $\left\langle n_{rs},x_{t}\right\rangle \geq0.$
However the formula for $da_{s}$ is unchanged if the normals are replaced by
their negatives and hence it is enough to construct the normals so that the
inner products $\left\langle n_{rs},x_{t}\right\rangle $ all have the same
sign. Taking note of the fact that $\left\langle a,b\times c\right\rangle
=\left\langle b,c\times a\right\rangle $ for vectors in $\mathbb{R}^{3}$ we
see that the choices
\begin{equation}
n_{rs}=\frac{x_{r}\times x_{s}}{\left\Vert x_{r}\times x_{s}\right\Vert
}=\frac{\widehat{x_{r}}\times\widehat{x_{s}}}{\left\Vert \widehat{x_{r}}%
\times\widehat{x_{s}}\right\Vert } \label{normal}%
\end{equation}
insures that the sign of $\left\langle n_{rs},x_{t}\right\rangle $ is
unchanged by cyclic permutation of the indices.\ Hence we can compute
\begin{align}
\cos da_{s}  &  =-\,\left\langle \left\langle n_{rs},n_{st}\right\rangle
\right\rangle \label{cos da}\\
&  =-\left\langle \left\langle \frac{\widehat{x_{r}}\times\widehat{x_{s}}%
}{\left\Vert \widehat{x_{r}}\times\widehat{x_{s}}\right\Vert },\frac
{\widehat{x_{s}}\times\widehat{x_{t}}}{\left\Vert \widehat{x_{s}}%
\times\widehat{x_{t}}\right\Vert }\right\rangle \right\rangle \nonumber\\
&  =-\frac{\left\langle \left\langle \widehat{x_{r}},\widehat{x_{s}%
}\right\rangle \right\rangle \left\langle \left\langle \widehat{x_{s}%
},\widehat{x_{t}}\right\rangle \right\rangle -\left\langle \left\langle
\widehat{x_{r}},\widehat{x_{t}}\right\rangle \right\rangle \left\langle
\left\langle \widehat{x_{s}},\widehat{x_{s}}\right\rangle \right\rangle
}{\left\Vert \widehat{x_{r}}\times\widehat{x_{s}}\right\Vert \left\Vert
\widehat{x_{s}}\times\widehat{x_{t}}\right\Vert }\nonumber\\
&  =\frac{\cos va_{rt}-\cos va_{rs}\cos va_{st}}{\sin va_{rs}\sin va_{st}%
}.\nonumber
\end{align}

This was a Euclidean computation but the hyperbolic geometry at the origin is
conformal to the Euclidean geometry. Hence the appropriate hyperbolic
computations, using tangent lines to geodesics and tangent planes to faces,
all with their tangencies at the origin, leads to that Euclidean computation.

We will study the dihedral angles of the $X$ at $x_{1}$ using the matrix
\[
\mathcal{M}_{-CDA}(X)=%
\begin{pmatrix}
1 & -\cos da_{23} & -\cos da_{24}\\
-\cos da_{32} & 1 & -\cos da_{34}\\
-\cos da_{42} & -\cos da_{43} & 1
\end{pmatrix}
.
\]
Here the subscript is a reminder that the matrix entries are the negatives of
cosines of dihedral angles.

Our computation of dihedral angles gives one statement of the hyperbolic law
of cosines.

\begin{definition}
Suppose that for $2\leq i,j\leq4$ we are given angles $\left\{  va_{ij}%
\right\}  $ and $\left\{  da_{ij}\right\}  $ with, for all $i,j,$
$va_{ij}=va_{ji},$ $da_{ij}=da_{ji},$ $va_{ii}=0,$ $da_{ii}=\pi;$ set
$VA_{ii}=0,$ $DA_{ii}=\pi.$ Define
\begin{align}
DA_{rs}  &  =\frac{\cos va_{rs}-\cos va_{tr}va_{ts}}{\sin va_{tr}\sin va_{ts}%
},\label{hyperbolic cosine}\\
VA_{rs}  &  =\frac{\cos da_{rs}+\cos da_{tr}\cos da_{ts}}{\sin da_{tr}\sin
da_{ts}}. \label{hyperbolic cosine 2}%
\end{align}

\end{definition}

The vertex angles and dihedral angles of a hyperbolic tetrahedron in
$\mathbb{RH}^{n}$are related by those two formulas.

\begin{lemma}
[Hyperbolic Law of Cosines]\label{hlc}If $X=\left\{  x_{1}\right\}  _{i=1}%
^{4}\subset\mathbb{RH}^{n}$ has vertex angles at $x_{1}$ $\left\{
va_{ij}\right\}  _{i,j=2}^{4}$ and dihedral angles $\left\{  da_{ij}\right\}
_{i,j=2}^{4},$ then, with the notation (\ref{hyperbolic cosine}) and
(\ref{hyperbolic cosine 2}),%
\[
\cos da_{ij}=DA_{ij},\text{ \ \ }\cos va_{ij}=VA_{ij}.,\text{ \ \ }2\leq
i,j\leq4.
\]

\end{lemma}

These formulas are classical, for instance \cite{Co}, \cite{J}. They can be
obtained by doing Euclidean vector computations. Alternatively they are
consequences of the spherical law of cosines applied to the triangle $\hat{X}
$ along with the relation between the geometry of $X$ and of $\hat{X}.$
Details are in \cite{Roe} and Chapter 5 of \cite{An} and the references there.

We separated the definition from the lemma because we can use the definitions
even if the $va_{ij}$ are not known to be data from a tetrahedron. If we have
a set of triangles $\left\{  T_{i}\right\}  _{i=2}^{4}$ in $\mathbb{RH}^{k}$
which satisfy the coherence conditions for assembly into a tetrahedron,
$\left\{  T_{i}\right\}  _{i=2}^{4}\rightrightarrows\,??$, then we can
construct the matrix $\mathcal{M}_{CV\!\!A}(\left\{  T_{i}\right\}  _{i=2}%
^{4}),$ Using that data in (\ref{hyperbolic cosine}) we can compute imputed
values of the $DA_{ij}.$ If $\left\{  T_{i}\right\}  _{i=2}^{4}%
\rightrightarrows X$ for a tetrahedron $X$ then those values will be cosines
of angles and satisfy $\left\vert D_{ij}\right\vert \leq1.$ In fact the
condition $\left\vert D_{ij}\right\vert \leq1$ is also sufficient for there to
be an $X;$ comparing the formula (\ref{hyperbolic cosine}) with
(\ref{cauchy 3}) we obtain

\begin{corollary}
If $\left\{  T_{i}\right\}  _{i=2}^{4}\rightrightarrows\,??$ and $DA_{ij}$ is
computed using the $va_{ij}$ values from $\mathcal{M}_{CV\!\!A}(\left\{
T_{i}\right\}  _{i=2}^{4})$ and (\ref{hyperbolic cosine}) then there is a
tetrahedron $X$ with $\left\{  T_{i}\right\}  _{i=2}^{4}\rightrightarrows X$
if and only if for some $i,j$ $\left\vert DA_{ij}\right\vert \leq1.$
\end{corollary}

That corollary is a special case of a more general observation. If we have
$\left\{  T_{i}\right\}  _{i=2}^{4}\rightrightarrows\,??$ then we can
construct $\mathcal{M}_{CV\!\!A}(\left\{  T_{i}\right\}  _{i=2}^{4}).$ Using
the data from that matrix and the formula (\ref{hyperbolic cosine}) we can
construct $\mathcal{M}_{-CDA}(\left\{  T_{i}\right\}  _{i=2}^{4}),$ a matrix
that would be $\mathcal{M}_{-CDA}(X)$ if we knew there were an $X.$ {}Both
$\mathcal{M}_{CV\!\!A}(\left\{  T_{i}\right\}  _{i=2}^{4})$ and $\mathcal{M}%
_{-CDA}(\left\{  T_{i}\right\}  _{i=2}^{4})$ are matrices to which Lemma
\ref{matrix condition} applies. The formulas (\ref{hyperbolic cosine}) and
(\ref{hyperbolic cosine 2}) give equivalences between the conditions from
Lemma \ref{matrix condition} which insure $\mathcal{M}_{CV\!\!A}(\left\{
T_{i}\right\}  _{i=2}^{4})\succcurlyeq0$ and the conditions from that lemma
which would insure, $\mathcal{M}_{-CDA}(\left\{  T_{i}\right\}  _{i=2}%
^{4})\succcurlyeq0.$

Those equivalences are an instance of a systematic duality between results
about vertex angles and results about dihedral angles$.$ Here is another
instance. Given a set of three angles $\Gamma=\left\{  \gamma_{2},\gamma
_{3},\gamma_{4}\right\}  $ it may or may not be true that those angles can
arise as the vertex angles of a nondegenerate trivalent vertex $\mathbb{RH}%
^{k}.$ If they can we say that $\Gamma$ are \textit{good vertex angles} and
write $\Gamma\in\mathcal{GVA}$. It also may of may not be true that those
angles can arise as a set of dihedral angles of a nondegenerate trivalent
vertex in $\mathbb{RH}^{k}.$ If so we call them \textit{good dihedral angles}
and write $\Gamma\in\mathcal{GDA}$. We have the following duality.

\begin{proposition}
Given the set of angles $\Gamma=\left\{  \gamma_{2},\gamma_{3},\gamma
_{4}\right\}  $ and the set $\Pi-\Gamma=$ $\left\{  \pi-\gamma_{2},\pi
-\gamma_{3},\pi-\gamma_{4}\right\}  ,$ then $\Gamma\in\mathcal{GVA}$ if and
only if $\Pi-\Gamma\in\mathcal{GDA}$.
\end{proposition}

Note that $\Pi-(\Pi-\Gamma)=\Gamma$ and hence the proposition is symmetric in
the two sets of angles.

\begin{proof}
First suppose $\Gamma\in\mathcal{GDA}$. In that case there is a tetrahedron
$X$ with dihedral angles $\Gamma$ and the negatives of the cosines of elements
of $\Gamma$ as entries in $\mathcal{M}_{-CDA}(X).$ We want to find a
tetrahedron $Y$ which shows $\Pi-\Gamma\in\mathcal{GVA}$. To do that consider
the matrix $\mathcal{M}_{-CDA}(X).$ Because it was constructed as a Gram
matrix (of three normal vectors) we know it is positive semidefinite. Hence
from Theorem \ref{Q 1} here is a tetrahedron $Y$ with $\mathcal{M}%
_{CV\!\!A}(Y)=\mathcal{M}_{-CDA}(X).$ Noting the identity $-\cos\gamma
=\cos(\pi-\gamma)$ we see that the entries of $\mathcal{M}_{CV\!\!A}(Y)$ are
the cosines of the angles in $\Pi-\Gamma.$ Hence this is the required $Y.$

Suppose now that $\Gamma\in\mathcal{GVA}$, specifically $\Gamma$ are the
vertex angles of a tetrahedron $X.$ As before the matrix $\mathcal{M}%
_{-CDA}(X)$ equals $\mathcal{M}_{CV\!\!A}(Y)$ for some tetrahedron $Y=\left\{
y_{i}\right\}  _{i=1}^{4}.$ Let $\Gamma_{Y}$ be the set of dihedral angles of
$Y$ and hence $\Gamma_{Y}\in\mathcal{GDA}$. We will be finished if we show
$\Pi-\Gamma_{Y}=\Gamma$ or, equivalently, that the cosines of elements of
$\Gamma.$are the negatives of the cosines of elements of $\Gamma_{Y}$. To see
this we analyze the construction of the cosines of the dihedral angles of $Y.$
The normal vectors to a face of $Y$ will be parallel to a vector $y_{i}^{\ast
}\times y_{j}^{\ast}$. By construction $y_{i}^{\ast}$ is parallel to
$x_{r}^{\ast}\times x_{s}^{\ast}$ where neither $r$ nor $s$ is equal to $i.$
Thus the normal vector will be parallel to ($x_{r}^{\ast}\times x_{s}^{\ast
})\times(x_{p}^{\ast}\times x_{q}^{\ast})$ where, also, neither $p$ nor $q$ is
equal to $j.$Because we only have three indices the normal must be parallel to
($x_{i}^{\ast}\times x_{k}^{\ast})\times(x_{j}^{\ast}\times x_{k}^{\ast})$
where $k$ is the index that is not $i$ or $j$. The normal is perpendicular to
the first factor and hence in the plane spanned by $x_{i}^{\ast}\ $and
$x_{k}^{\ast}. $ Similarly, looking at the second factor we see the normal is
in the span of $x_{i}^{\ast}\ $and $x_{k}^{\ast}.$ Hence it is in the
intersection of those spans which is the span of $x_{k}^{\ast}$ (recall that
we assumed nondegeneracy). The normal vector is a unit vector in the span
$x_{k}^{\ast} $ and hence is $\pm x_{k}^{\ast}.$ We want the inward facing
normal and hence it is $x_{k}^{\ast}.$The inner products of these normal
vectors are, by definition, the negatives of the cosines of the dihedral
angles of $Y;$ they are also, by inspection, the cosines of the vertex angles
of $X.$ Thus the two are equal, as we wanted.
\end{proof}

(The previous results were stated and proved in the language of vertex angles
and dihedral angles of a hyperbolic tetrahedron. However the results can also
be formulated and proved in the language of spherical geometry. The previous
result is the relationship between the angles in the spherical triangle
$\hat{X}=\{\hat{x}_{2},\hat{x}_{3},\hat{x}_{4}\}\subset\mathbb{S}_{2}$ and the
angles in its polar dual. Alternatively the results are about the Euclidean
geometry of $\mathbb{R}^{3}$ associated with the cross product.)

Using the details of the previous proof lets us give a version of Theorem
\ref{Q 1} for dihedral angles. Let $\mathcal{L}=\left(  L_{ij}\right)
_{i,j=2}^{4}$ be the symmetric matrix with real entries given by
\[
\mathcal{L=}%
\begin{pmatrix}
1 & L_{23} & L_{24}\\
L_{32} & 1 & L_{24}\\
L_{42} & L_{42} & 1
\end{pmatrix}
\]

\begin{theorem}
\label{Q 1 DA}The matrix $\mathcal{L}$ is the matrix $\mathcal{M}_{-CDA}(X)$
of a tetrahedron $X$ in $\mathbb{RH}^{k}$ if and only $\mathcal{L}%
\succcurlyeq0.$

\begin{proof}
We saw in the previous proof that for any $X$ we have $\mathcal{M}%
_{-CDA}(X)\succcurlyeq0.$ Suppose now we have $\mathcal{L}\succcurlyeq0$ then
by Theorem \ref{Q 1} $\mathcal{L}=\mathcal{M}_{CV\!\!A}(Y)$ for a tetrahedron
$Y$ in $\mathbb{RH}^{k}$. We now follow the pattern of the previous proof.
From $\mathcal{M}_{CV\!\!A}(Y)$ we obtain $\mathcal{M}_{-CDA}(Y),$ and then,
by Theorem \ref{Q 1}, $\mathcal{M}_{-CDA}(Y)=\mathcal{M}_{CV\!\!A}(Z)$ for
some $Z.$ Now consider the matrix $\mathcal{M}_{-CDA}(Z).$ The argument in the
previous proof shows that this type of repeated passage from vertex angles to
dihedral angles reproduces the original data; in particular $\mathcal{L}%
=\mathcal{M}_{-CDA}(Z)$ which shows that $\mathcal{L}$ is of the required form.
\end{proof}
\end{theorem}

If we suppose that the $\left\vert L_{ij}\right\vert \leq1,$ for instance if
the $\left\{  L_{ij}\right\}  $ are the negatives of cosines of angles, then
we can apply Lemma \ref{matrix condition} to obtain various statements
equivalent to $\mathcal{L}\succcurlyeq0.$ If $\mathcal{L}=\left(
L_{ij}\right)  $ with $L_{ij}=-\cos da_{ij}$ for some angles $da_{ij}=da_{ji}$
and $da_{ii}=\pi$ then those angles are the dihedral angles at the vertex of a
tetrahedron in $\mathbb{RH}^{k}$ if and only if $\mathcal{L}\succcurlyeq0$ and
that happens if and only if
\begin{equation}
(\cos da_{34}+\cos da_{23}\cos da_{42})^{2}\leq(1-\cos^{2}da_{23})(1-\cos
^{2}da_{42}) \label{trig da}%
\end{equation}
(or the similar result with a different distinguished index). This condition
and the earlier (\ref{trig va}) are analogs of (\ref{hopw}) for the values of
$\operatorname*{kos}$ at the vertex of a tetrahedron in $\mathbb{CH}^{n}.$
However in this case the fact that we are working with real numbers allows
very substantial simplifications.

\begin{proposition}
\label{conditions}

\begin{enumerate}
\item If $\Gamma=\left\{  \alpha,\beta,\gamma\right\}  \subset\left(
0,\pi\right)  $ then $\Gamma\in\mathcal{GVA}$ if and only if
\begin{equation}
\alpha\leq\beta+\gamma. \label{tia}%
\end{equation}
That inequality is sometimes called the "triangle inequality for angles".

\item If $\Gamma=\left\{  A,B,C\right\}  \subset\left(  0,\pi\right)  $ then
$\Gamma\in\mathcal{GDA}$ if and only if%
\begin{equation}
\pi\leq A+B+C \label{tid}%
\end{equation}

\end{enumerate}
\end{proposition}

\begin{proof}
We start with (\ref{trig va}) and then pass to equivalent formulations:%

\begin{align*}
(\cos\alpha-\cos\beta\cos\gamma)^{2}  &  \leq(1-\cos^{2}\beta)(1-\cos
^{2}\gamma)\\
\left\vert \cos\alpha-\cos\beta\cos\gamma\right\vert  &  \leq\sin\beta
\sin\gamma\\
-\sin\beta\sin\gamma+\cos\beta\cos\gamma &  \leq\cos\alpha\leq\sin\beta
\sin\gamma+\cos\beta\cos\gamma\\
\cos\left(  \beta+\gamma\right)   &  \leq\cos\alpha\leq\cos\left\vert
\beta-\gamma\right\vert
\end{align*}

If $\gamma+\beta<\pi$ then the three angles in the previous line are in the
range $\left(  0,\pi\right)  $ where the cosine is monotone decreasing. In
that case the first inequality gives $\alpha\leq\beta+\gamma.$ In the other
case we have $\alpha\leq\pi\leq\beta+\gamma.$ In both cases we have (\ref{tia}).

For the dihedral angles we first note that $-\cos A=\cos\left(  \pi-A\right)
.$ We then follow the previous pattern and obtain
\[
\cos\left(  B+C\right)  \leq\cos\left(  \pi-A\right)  \leq\cos\left\vert
B-C\right\vert .
\]
Following the same analysis as before we obtain $\pi-A<B+C$ which is
equivalent to (\ref{tid}).
\end{proof}

These are elementary results in real hyperbolic geometry and certainly can be
given much more direct proofs. On the other hand this approach gives
interesting insight into how results such as Theorem \ref{Q 1} can be seen as
extensions to complex hyperbolic geometry of basic facts from real hyperbolic geometry.

\subsection{Factorization}

The expressions $\det\mathcal{M}_{-CDA}(X)$ and $\det\mathcal{M}%
_{CV\!\!A}(X)(X)$ arise in hyperbolic trigonometry where they are called
amplitudes. Among their many properties are trigonometric factorizations. The
formulas can be obtained by trigonometric analysis \cite[pg. 107]{F},
\cite[(2.88) - (2,94)]{J} or, as Roeder points out \cite{Roe}, by direct
computation with complex exponentials.

\begin{proposition}
Set
\begin{align*}
s  &  =(va_{23}+va_{24}+va_{34})/2\\
S  &  =(da_{23}+da_{24}+da_{34})/2
\end{align*}
then%
\begin{align}
&  1+2\cos va_{23}\cos va_{24}\cos va_{34}-\cos^{2}va_{23}-\cos^{2}%
va_{24}-\cos^{2}va_{34}\label{one}\\
&  \qquad\qquad=4\sin\left(  s\right)  \sin(s-va_{23})\sin(s-va_{24}%
)\sin(s-va_{34})\nonumber
\end{align}

\end{proposition}

and%

\begin{align}
&  1-2\cos da_{23}\cos da_{24}\cos da_{34}-\cos^{2}da_{23}-\cos^{2}%
da_{24}-\cos^{2}da_{34}\label{two}\\
&  \qquad\qquad=-4\cos\left(  S\right)  \cos(S-da_{23})\cos(S-da_{24}%
)\cos(S-da_{34})\nonumber
\end{align}
Given a tetrahedron $X\subset\mathbb{RH}^{n}$ the left hand side of
(\ref{one}) is $\det\mathcal{M}_{CV\!\!A}(X)$ and hence by (\ref{det}) must be
positive. We can use that as a starting point and study the signs of the
individual factors and derive (\ref{tia}). Similarly in \cite{Roe} Roeder
analyzes the sign of the factors in (\ref{two}) and obtains (\ref{tid}) for
acute angles $A,B,C.$

\section{Final Comments\label{final}}

\textbf{Cayley Equations: }If we replace the trigonometric variables in the
statements $\det\mathcal{M}_{-CDA}(X)=0$ and $\det\mathcal{M}_{CV\!\!A}%
(X)(X)=0$ with algebraic variables we obtain two algebraic equations:%
\begin{align}
1-2xyz-x^{2}-y^{2}-z^{2}  &  =0,\nonumber\\
p(x,y,z)=1+2xyz-x^{2}-y^{2}-z^{2}  &  =0. \label{22}%
\end{align}
These equations were studied by Cayley in his classic study of cubic equations
and sometimes carry his name \cite{H}. For us the region in $\mathbb{R}^{3}$
where the variables have absolute value at most one and $p(x,y,z)>0$
parameterizes nondegenerate tetrahedra in $\mathbb{RH}^{3}.$ The boundary
surface $\Omega,$ where $p(x,y,z)=0,$ correspond to degenerate tetrahedra. The
smooth points of $\Omega$ correspond to simple degenerations, degenerate
tetrahedra that become nondegenerate when a single vertex is moved a small
amount. The singular points of $\Omega$ correspond to more complicated,
nongeneric, degeneracies. For instance, let $T$ be a triangle $\left\{
w,y,z\right\}  $ in the ball model of $\mathbb{RH}^{3},$ $\mathbb{RB}_{3},$
which is in the plane $\mathbb{RH}^{2}$ specified by the vanishing of the
third coordinate and which has the origin of that plane in its interior. Form
tetrahedra $X_{\varepsilon}$ by adjoining a fourth vertex, which will be the
distinguished vertex $x_{1},$ with Euclidean coordinates $\left(
0,0,\varepsilon\right)  $ for a small positive $\varepsilon.$ The
$X_{\varepsilon}$ are proper tetrahedra but the limiting $X_{0}$ is a
degenerate tetrahedron$.$ The vertices of $X_{0}$ are all in the plane
$\mathbb{RH}^{2},$ and include the origin of that plane as the distinguished
vertex. Thus for $X_{0}$ the values of $\operatorname*{kos}_{1}$ are the
cosines of the angles formed by connecting the origin to the other vertices.
We are in a plane so those angles sum to 2$\pi.$ Thus the corresponding
$\left(  x,y,z\right)  $ values are $\left(  \cos\alpha,\cos\beta,\cos\left(
2\pi-\alpha-\beta\right)  \right)  $ for some angles $\alpha$ and $\beta.$
These values satisfy (\ref{22}), as can be checked by noting that equality
holds in (\ref{cauchy 3}). That point is a smooth point on the surface
$\Omega$ and the degeneracy of $X_{0}$ can be removed by moving the vertex at
the origin slightly to obtain an $X_{\varepsilon}.$

The singular points of $\Omega$ are the points ($\pm1,\pm1,\pm1)$ with an even
number of minus signs. The corresponding tetrahedra have four points on a
single real geodesic. Those tetrahedra are have nongeneric degeneracy, they
remain degenerate if any of the points is moved slightly.

Some related discussion is in \cite{H}.

\textbf{Larger Polyhedra: }A polyhedron in $\mathbb{RH}^{n}$ is determined by
a set of vertex points together with combinatoric data telling which vertices
are connected by edges, which edges bound faces, etc. For tetrahedra the
combinatorics are trivial, every pair of vertices bounds an edge, every triple
of edges bound a face. Hence the study of tetrahedra is essentially equivalent
to the study of the geometry of four point sets. Also, for four point sets we
obtained relatively clean results because the dimensions were small.

For larger polyhedra we obtain information that is more complicated. This is
seen for instance in Theorem \ref{v3}. Also the information from Theorem
\ref{general} and related results says nothing about the polyhedral
combinatorics the set might have. It would be interesting if there were a
natural way to encode the polyhedral structure as a layer of structure in the
Hilbert spaces $DA(X).$

The appropriate definition of a polyhedron in $\mathbb{CH}^{n}$ is not clear.
An indication that the situation is very different from $\mathbb{RH}^{n}$ can
already be seen with triangles. Any three points in $\mathbb{RH}^{n}$ are
contained in a totally geodesically embedded $\mathbb{RH}^{2}$ and that allows
a natural definition of the face bounded by the geodesics connecting the
vertices. However there is no similar construction for triples of points in
$\mathbb{CH}^{n}$, in particular there is no natural notion of the face of a
triangle. This shows up, for instance, in the fact that symplectic area of a
triangle, a substitute for classically defined area, is defined in a way that
is independent of the choice of real two manifold joining the geodesic edges.

(Some thought has been given to defining and describing polyhedra in
$\mathbb{C}^{n}$ \cite{Co} but that work does not seem directly relevant here.)

\textbf{Vertices at Infinity:} The study of tetrahedra in real hyperbolic
space, and to some extent in complex hyperbolic space, is not restricted to
classical bounded tetrahedra but also includes consideration of ideal
tetrahedra, tetrahedra with one or more vertices in the ideal boundary (i.e.
the "sphere at infinity", $\partial\mathbb{B}_{n}$). Although some of the
previous discussion extends to those contexts it is not clear if there are
spaces similar to the $DA(X)$ which are useful in studying those geometries.
It may be that the adjoining ideal boundary points to hyperbolic space is
analogous to adjoining vectors of infinite norm to the $DA(X).$

There is analysis of congruence of finite sets in the closure $\overline
{\mathbb{CH}^{n}}$ in several places including \cite{Go}, \cite{HS}, \cite{G},
and \cite{CG}.

\textbf{The Physics Literature:} The question of characterizing triples of
triangles in $\mathbb{RH}^{3}$ which can be assembled as part of a tetrahedron
is also studied in the physics literature, sometimes with the name "closure
questions", for instance \cite{CL}, \cite{BDGL}, \cite{HHR}. In contrast to
the work here, those papers make substantial use of the descriptive and
analytical properties of the automorphism group.

\textbf{Other Hilbert Spaces, Other Geometries:} The relation between RKHSI
and geometry extends to general RKHSI and includes relations to other
classical geometries. We have focused on the relation between the $DA$ spaces
and hyperbolic geometry \ There are similar relations between the
Segal--Bargmann--Fock spaces and the Hermitian geometry of $\mathbb{C}^{n},$
and between the Hilbert spaces of spin coherent states \cite{BZ}, and the
geometry of complex spheres and projective spaces.

More general relations between geometry and spaces such as $DA(X)$ for $X$ in
$\mathbb{RH}^{n}$ and $\mathbb{CH}^{n}$ are suggested by the work in \cite{M}.

\end{document}